\documentclass[11pt,c]{article}

\usepackage{amsmath}
\usepackage{amsfonts}
\usepackage{mathrsfs}
\usepackage{amssymb}
\usepackage{graphicx}
\usepackage[english]{babel}
\usepackage{a4wide}
\usepackage{times}
\usepackage{amsbsy}

\newtheorem{theorem}{Theorem}

\newtheorem{definition}[theorem]{Definition}

\newtheorem{lemma}[theorem]{Lemma}

\newtheorem{remark}[theorem]{Remark}

\newenvironment{proof}[1][Proof]{\noindent\textbf{#1.} }{\ \rule{0.5em}{0.5em}}
\numberwithin{equation}{section}

\begin{document}

\begin{center}

{\large {\bf
MIXED {$\boldsymbol{hp}$}-FINITE ELEMENT METHOD FOR LINEAR \\
ELASTICITY WITH WEAKLY IMPOSED SYMMETRY III: \\
STABILITY ANALYSIS IN 3D
\\} }

\vspace*{.25cm} 

Weifeng Qiu and Leszek Demkowicz

\vspace*{.25cm}

{\bf Institute for
Computational Engineering and Sciences\\
The University of Texas at Austin, Austin, TX 78712, USA }
\end{center}

\begin{abstract}
The paper presents a generalization of Arnold-Falk-Winther elements 
for three dimensional linear elasticity, to meshes 
with elements of variable order. The generalization is straightforward 
but the stability analysis involves a non-trivial modification of 
involved interpolation operators. The analysis addresses only the
$h$-convergence. 
\end{abstract}



\section{Introduction}
Linear elasticity is a classical subject, and it has been studied for 
a long time. The paper focuses on the so-called dual--mixed formulation
with weakly imposed symmetry that may be derived by considering stationary
points of the generalized Hellinger-Reissner functional \cite{ODENREDDY:1976:VMTM}.
We restrict ourselves to the static 
case only and, for the sake of simplicity, we assume that the body is 
fixed on the whole boundary. 
We look for stress tensor $\sigma\in
H(\text{div},\Omega;\mathbb{M})$, displacement vector $u\in L^{2}(\Omega;\mathbb{V})$, and 
infinitesimal rotation $p\in L^{2}(\Omega;\mathbb{K})
$ satisfying
\begin{gather}
\int_{\Omega}(A\sigma:\tau+\text{div}\tau\cdot u+\tau:p)
d\boldsymbol{x}=0,\text{ \ }\tau\in H(\text{div},\Omega;\mathbb{M})
,\label{weak_symm_formula_continuous}\\
\int_{\Omega}\text{div}\sigma\cdot vd\boldsymbol{x}=\int_{\Omega}f\cdot vd\boldsymbol{x},
\text{ \ }v\in L^{2}(\Omega;\mathbb{V})  ,\nonumber\\
\int_{\Omega}\sigma:qd\boldsymbol{x}=0,\text{\ \ }q\in L^{2}(\Omega;\mathbb{K}).\nonumber
\end{gather}
The first equation represents a relaxed form of the Hooke's law combined with 
Cauchy geometrical
relations, the second one represents the equilibrium equations (in a strong form),
and the third one enforces the symmetry of the stress tensor. 
We refer to the next section for a detailed 
description of energy spaces: $H(\text{div},\Omega;\mathbb{M})$, $L^{2}(\Omega;\mathbb{V})$ 
and $L^{2}(\Omega;\mathbb{K})$. The operator $A$ denotes the compliance tensor (operator)
mapping stress tensor into strain tensor. The operator is bounded, symmetric, 
uniformly positive definite, and it preserves the symmetry of the tensor.

The traditional motivation for studying the formulation~(\ref{weak_symm_formula_continuous}) 
comes from  handling nearly incompressible materials. Our interest in the subject
stems from a study of a class of (visco)elastic vibration problems for structures with 
large material contrast, see~\cite{QD:2009:MMEW} for a motivating example.

A number of authors have developed approximation schemes based on 
formulation~(\ref{weak_symm_formula_continuous}), among others see
\cite{AFW:2006:ECH,AFW:2007:MMEW,Falk:2008:FME,AT:1979:EFEE,ABD:1984:MFEPE,AF:1988:NMFEE,
FarhloulFortin:1997:DHES,FdV:1975:SFA,Morley:1989:MFEE,SteinRolfes:1990:SOMFEPE,
Stenberg:1986:COMME,Stenberg:1988:FMME,Stenberg:1988:TLOMME}. 
For a brief description of these methods, 
we refer to the introduction in \cite{AFW:2007:MMEW}. 
We also refer to the recent work of 
Cockburn, Gopalakrishnan 
and Guzman \cite{CGG:2009:NEEMW} who
have developed a new mixed method for linear elasticity using a hybridized 
version of~(\ref{weak_symm_formula_continuous}).

The work presented in this paper is based on the mixed finite element methods
developed by Arnold, Falk and Winther  in \cite{AFW:2006:ECH,Falk:2008:FME,AFW:2007:MMEW}. 

The ultimate goal of this work is to lay down theoretical foundations for,
and implement a fully automatic $hp$-adaptive Finite Element (FE) method based on
a generalization of the AFW element to meshes with variable order.
The generalization builds on the exact grad-curl-div sequence that holds for $hp$ meshes,
see \cite{Demkowicz:2006:HPAFE,Demkowicz:2007:HPAFE2} and it is rather
straightforward. The formulation is easily accommodated in a general $hp$ code
supporting the exact sequence.

The $h$ convergence analysis presented in \cite{AFW:2007:MMEW} for 
meshes with arbitrary but uniform polynomial order, does not however
generalize immediately to elements with variable order.

With the proof of $p$ and $hp$ convergence as an ultimate goal,
our initial efforts start with a less ambitious goal of proving first stability
and convergence for uniform $h$-refinements of meshes of variable order.

At the first glance, a generalization of the techniques from \cite{AFW:2007:MMEW}
seems to be easy.
But, as we have shown in \cite{QD:2009:MMEW}, the (natural generalization of) canonical 
projection operators defined in \cite{AFW:2006:ECH} do not commute with divergence 
operator on meshes with variable order, a property essential in the proof of discrete 
stability. We have resolved this problem by invoking the Projection Based (PB) interpolation 
operators in \cite{QD:2009:MMEW}. Unfortunately, the PB operators do not commute with 
an algebraic
operator $S_{n-2}$, introduced in \cite{AFW:2006:ECH}, another essential
construction in the AFW proof of discrete stability. We have resolved the problem by
designing a new, special operator $\tilde{W}_{h}$ in \cite{QD:2009:MMEW}
that satisfies the commutativity property, as needed.
Unfortunately, we managed to prove
well-definedness of operator $\tilde{W}_{h}$ only for polynomial orders 
$0 \leq p \leq 3$, and only for two space dimensions, see \cite{QD:2009:MMEW}.
In this contribution, we resolve
the problem by designing new PB operators and a a variant of $\tilde{W}_{h}$,
a new operator 
$\overline{\Pi}_{\tilde{r},h}^{1,-}$, 
discussed in the text.

An outline of the paper is the follows. Section 2 introduces notations.
In Section 3, we define the involved finite element spaces.
Section 4 recalls the mixed formulation of linear elasticity 
with weakly imposed symmetry and states the Brezzi conditions for the stability.
In Section 5, we establish all technical results needed for proving the stability. 
We construct the new 
PB operators and the operator $\overline{\Pi}_{\tilde{r},h}^{1,-}$. 
Finally, in Section 6, we prove the Brezzi conditions.

\section{Notations}
In this section, we introduce some basic notations.
We define $\mathbb{M}$ to be the space of $3\times 3$ real matrices, and 
$\mathbb{V}$ to be $\mathbb{R}^{3}$. For any $3\times 3$ real matrices 
$A,B$, we define 
$$
A:B=\text{tr}(AB^{\top}).
$$ 
We denote by $\mathbb{S}$ 
and $\mathbb{K}$ the subspaces of symmetric and anti-symmetric matrices 
in $\mathbb{R}^{3\times 3}$. Each anti-symmetric matrix can be identified 
with a vector in $\mathbb{V}$ given by the mapping $\text{vec}:\mathbb{K}
\rightarrow\mathbb{V}$:
\begin{equation*}
\text{vec}\left[\begin{array}{ccc}
0 & -v_{3} & v_{2}\\
v_{3} & 0 & -v_{1}\\
-v_{2} & v_{1} & 0
\end{array}\right]=
\left[\begin{array}{c}
v_{1}\\
v_{2}\\
v_{3}
\end{array}\right].
\end{equation*}
$\Omega$ is a domain in $\mathbb{R}^{3}$. For any vector space $\mathbf{X}$ with 
inner product, we denote by $L^{2}(\Omega;\mathbf{X})$
the space of square-integrable vector fields on $\Omega$ with values in $\mathbf{X}$.  
In the paper, $\mathbf{X}$ will be $\mathbb{R}$, $\mathbb{V}$, $\mathbb{M}$, or $\mathbb{K}$.
When $\mathbf{X}=\mathbb{R}$, we will write $L^{2}(\Omega)$. The norm associated 
with $L^{2}(\Omega; \mathbf{X})$, denoted by $\Vert\cdot\Vert_{L^{2}(\Omega;\mathbf{X})}$, 
is obtained by taking the
square root of the sum of (squared) $L^2$ norms of individual  components of the  
vector fields 
on $\Omega$. 

Notice that, for scalar-valued functions,
norm $\Vert\cdot\Vert_{L^{2}(\Omega)}$ coincides with 
the standard $L^{2}$-norm. The corresponding Sobolev space of order 
$m$, which is the subspace of $L^{2}(\Omega;\mathbf{X})$ consisting of functions with 
all partial derivatives of order less than or equal to $m$ in $L^{2}(\Omega;\mathbf{X})$, 
is denoted 
by $H^{m}(\Omega;\mathbf{X})$. The norm associated with $H^{m}(\Omega;\mathbf{X})$, 
denoted by $\Vert\cdot\Vert_{H^{m}(\Omega;\mathbf{X})}$, 
equals to the square root of the sum of (squared) $L^2$-norms  of all
partial derivatives with order less than or equal to $m$, for all components of 
vector fields  on $\Omega$. When $\mathbf{X}=\mathbb{R}$, $\Vert\cdot\Vert_{H^{m}(\Omega)}$ 
coincides with the standard $H^{m}$-norm for scalar-valued functions. 

The spaces $H(\text{curl},\Omega), H(\text{div},\Omega)$ are defined by 
\begin{align*}
& H(\text{curl},\Omega)=\{\boldsymbol{u}\in L^{2}(\Omega;\mathbb{V}):\text{curl}
\boldsymbol{u}\in L^{2}(\Omega;\mathbb{V})\}\\
& H(\text{div},\Omega)=\{\boldsymbol{v}\in L^{2}(\Omega;\mathbb{V}):\text{div}
\boldsymbol{v}\in L^{2}(\Omega)\}
\end{align*}
with the norms, 
\begin{align*}
& \Vert\boldsymbol{u}\Vert_{H(\text{curl},\Omega)}=(\Vert\boldsymbol{u}\Vert_{L^{2}
(\Omega;\mathbb{V})}^{2}+\Vert\text{curl}\boldsymbol{u}\Vert_{L^{2}(\Omega;\mathbb{V})}^{2})^{1/2}, 
\quad \boldsymbol{u}\in H(\text{curl},\Omega)\\
& \Vert\boldsymbol{v}\Vert_{H(\text{div},\Omega)}=(\Vert\boldsymbol{v}\Vert_{L^{2}
(\Omega;\mathbb{V})}^{2}+\Vert\text{div}\boldsymbol{v}\Vert_{L^{2}(\Omega)}^{2})^{1/2}, 
\quad \boldsymbol{v}\in H(\text{div},\Omega).
\end{align*}

We extend the definitions of $\nabla$ to $\mathbb{V}$-valued 
functions, $\text{curl}$ and $\text{div}$ to $\mathbb{M}$-valued 
functions by applying these operators row-wise.
The space $H(\text{curl},\Omega;\mathbb{M}),H(\text{div},\Omega;\mathbb{M})$ 
are defined by
\begin{align*}
& H(\text{curl},\Omega;\mathbb{M})=\{\sigma\in L^{2}(\Omega;\mathbb{M}):\text{curl}
\sigma\in L^{2}(\Omega;\mathbb{M})\} \\
& H(\text{div},\Omega;\mathbb{M})=\{\sigma\in L^{2}(\Omega;\mathbb{M}):\text{div}
\sigma\in L^{2}(\Omega;\mathbb{V})\}
\end{align*}
with the norms,
\begin{align*}
& \Vert\sigma\Vert_{H(\text{curl},\Omega;\mathbb{M})}=(\Vert\sigma\Vert_{L^{2}
(\Omega;\mathbb{M})}^{2}+\Vert\text{curl}\sigma\Vert_{L^{2}(\Omega;\mathbb{M})}^{2})^{1/2} \\
& \Vert\sigma\Vert_{H(\text{div},\Omega;\mathbb{M})}=(\Vert\sigma\Vert_{L^{2}
(\Omega;\mathbb{M})}^{2}+\Vert\text{div}\sigma\Vert_{L^{2}(\Omega;\mathbb{V})}^{2})^{1/2}.
\end{align*}

$\mathcal{P}_{r}(\Omega)$ denotes the space of polynomials on $\Omega$ with degree 
less than or equal to $r$. When $r$ is a negative integer, $\mathcal{P}_{r}(\Omega)=\{0\}$. 
$\mathcal{P}_{r}(\Omega;\mathbb{V})=[\mathcal{P}_{r}(\Omega)]^{3}$. 
Throughout this paper, we assume that $r$ is a nonnegative integer.

\section{Finite element spaces}

\subsection{Finite element spaces on a single tetrahedron}

Let $T$ be an arbitrary tetrahedron in $\mathbb{R}^{3}$. We denote by $\triangle_{k}(T)$, the 
union of $k$-dimensional subsimplexes of $T$. We denote by $\triangle (T)$, the 
union of all subsimplexes of $T$. 

For any $r\in\mathbb{Z}_{+}:=\{n\in\mathbb{Z}:n\geq 0\}$, we introduce
\begin{equation*}
\mathcal{P}_{r}\Lambda^{3}(T):= \mathcal{P}_{r}(T),
 \mathcal{P}_{r}\Lambda^{2}(T):=\mathcal{P}_{r}(T;\mathbb{V}),
\end{equation*}
\begin{equation*}
\mathring{\mathcal{P}}_{r}\Lambda^{2}(T):=\{\omega\in\mathcal{P}_{r}\Lambda^{2}(T)
:\forall F\in\triangle_{2}(T), \omega\cdot\boldsymbol{n}|_{F}=0\},
\end{equation*}
\begin{equation*}
\mathcal{P}_{r}^{-}\Lambda^{2}(T):=\mathcal{P}_{r-1}(T;\mathbb{V})
+\boldsymbol{x}\mathcal{P}_{r-1}(T),
\end{equation*}
\begin{equation*}
\mathring{\mathcal{P}}_{r}^{-}\Lambda^{2}(T):=\{\omega\in\mathcal{P}_{r}^{-}\Lambda^{2}(T)
:\forall F\in\triangle_{2}(T), \omega\cdot\boldsymbol{n}|_{F}=0\},
\end{equation*}
\begin{equation*}
\mathcal{P}_{r}^{-}\Lambda^{1}(T):=\mathcal{P}_{r-1}(T;\mathbb{V})
+\boldsymbol{x}\times\mathcal{P}_{r-1}(T;\mathbb{V}),
\end{equation*}
\begin{equation*}
\mathring{\mathcal{P}}_{r}^{-}\Lambda^{1}(T):=\{\omega\in\mathcal{P}_{r}^{-}\Lambda^{1}(T):
\forall F\in\triangle_{2}(T),\omega-(\omega\cdot\boldsymbol{n})\boldsymbol{n}|_{F}=0\}.
\end{equation*}
\begin{equation}
\mathring{\mathcal{P}}_{r}\Lambda^{0}(T):=\{u\in\mathcal{P}_{r}(T):u|_{\partial T}=0\}.
\label{FEM_spaces_uniform_order}
\end{equation}
Here $\boldsymbol{n}$ is a normal  unit  vector on $F$.
For $F$, an arbitrary face of $T$, we introduce 
\begin{align}
\mathcal{P}_{r}^{-}\Lambda^{1}(F):=\mathcal{P}_{r-1}(F;\mathbb{R}^{2})+\boldsymbol{y}
\mathcal{P}_{r-1}(F).
\label{FEM_space_face}
\end{align}
Here $\boldsymbol{y}$ denote any orthogonal coordinates on $F$.
In \cite{AFW:2006:ECH,Falk:2008:FME}, spaces in (\ref{FEM_spaces_uniform_order},\ref{FEM_space_face}) 
are defined in the language of exterior calculus. Here we rewrite them 
in the standard language of calculus. Please refer to \cite{AFW:2006:ECH} 
and \cite{Falk:2008:FME} for
a detailed correspondence.

We denote by $\tilde{r}$ a mapping from $\triangle (T)$ 
to $\mathbb{Z}_{+}$ such that if $e,f\in\Delta(T)$ and $e\subset f$ then
$\tilde{r}(e)  \leq\tilde{r}(f)$.
We introduce now formally the FE spaces of variable order.
\bigskip

\begin{definition}
\label{FEM_spaces_single_triangle_variable_order}
\[
\mathcal{P}_{\tilde{r}}\Lambda^{3}(T):=
\mathcal{P}_{\tilde{r}(T)}\Lambda^{3}(T)
=\mathcal{P}_{\tilde{r}(T)}(T),
\]
\[
\mathcal{P}_{\tilde{r}}\Lambda^{2}(T):=
\{\omega\in\mathcal{P}_{\tilde{r}(T)}\Lambda^{2}(T):
\forall F\in\triangle_{2}(T) \text{, } \omega\cdot \boldsymbol{n}|_{F}
\in\mathcal{P}_{\tilde{r}(F)}(F)\},
\]
\[
\mathcal{P}_{\tilde{r}}^{-}\Lambda^{2}(T):=
\{\omega\in\mathcal{P}_{\tilde{r}(T)}^{-}\Lambda^{2}(T):
\forall F\in\triangle_{2}(T) \text{, } \omega\cdot \boldsymbol{n}|_{F}
\in\mathcal{P}_{\tilde{r}(F)-1}(F)\},
\]
\[
\mathcal{P}_{\tilde{r}}^{-}\Lambda^{1}(T):=
\{\omega\in\mathcal{P}_{\tilde{r}(T)}^{-}\Lambda^{1}(T):
\forall F\in\triangle_{2}(T) \text{, } \omega-(\omega\cdot\boldsymbol{n})\boldsymbol{n}|_{F}
\in\mathcal{P}_{\tilde{r}(F)}^{-}\Lambda^{1}(F);
\]
$
\forall \boldsymbol{t}\in\triangle_{1}(T), 
\omega\cdot\boldsymbol{t}\in \mathcal{P}_{\tilde{r}(e)-1}(e)\}
$
where $\boldsymbol{t}$ is a tangential vector on $e$. 
\end{definition}

\begin{remark}
In the definition of $\mathcal{P}_{\tilde{r}}^{-}\Lambda^{1}(T)$, 
$\omega-(\omega\cdot\boldsymbol{n})\boldsymbol{n}|_{F}$ is a tangential vector field 
on $F$. So $\omega-(\omega\cdot\boldsymbol{n})\boldsymbol{n}|_{F}$ can be considered
as a two-component vector field.
\end{remark}

\bigskip

\begin{definition}
\label{FEM_spaces_single_triangle_variable_order_vector}
We define $\mathcal{P}_{\tilde{r}}\Lambda^{3}(T;\mathbb{V}):=\mathcal{P}_{\tilde{r}(T)}(T;\mathbb{V})$. 
We also define $\mathcal{P}_{\tilde{r}}\Lambda^{2}(T;\mathbb{V})$,$\mathcal{P}_{\tilde{r}}^{-}\Lambda^{2}
(T;\mathbb{V})$, and $\mathcal{P}_{\tilde{r}}^{-}\Lambda^{1}(T;\mathbb{V})$ as matrix-valued polynomial spaces
whose rows stay in $\mathcal{P}_{\tilde{r}}\Lambda^{2}(T)$, $\mathcal{P}_{\tilde{r}}^{-}\Lambda^{2}(T)$, 
and $\mathcal{P}_{\tilde{r}}^{-}\Lambda^{1}(T)$ respectively.
\end{definition}

\begin{remark}
Finite element spaces in definitions \ref{FEM_spaces_single_triangle_variable_order} 
and \ref{FEM_spaces_single_triangle_variable_order_vector} are the same as those 
introduced in 
\cite{QD:2009:MMEW} for $n=3$. In this paper, we just rewrite them in the language of 
standard calculus.
\end{remark}

\subsection{Finite element spaces on a bounded polyhedral domain}

Let $\mathcal{T}_{h}$ be a tetrahedral mesh. Here $h$ represents the biggest diameter of 
tetrahedrons in $\mathcal{T}_{h}$. We extend the map $\tilde{r}$ to a mapping from 
$\triangle (\mathcal{T}_{h})$ to $\mathbb{Z}_{+}$ such that if $e\subset f$, then 
$\tilde{r}(e)\leq \tilde{r}(f)$. For any $T\in\triangle_{3}(\mathcal{T}_{h})$, 
the restriction of $\tilde{r}$ to $\triangle (T)$ is represented with the same 
symbol $\tilde{r}$. 
We denote by $\triangle_{k}(\mathcal{T}_{h})$ the union of $k$-dimensional 
subsimplexes of $\mathcal{T}_{h}$, 
and by $\triangle (\mathcal{T}_{h})$ the union of all subsimplexes of $\mathcal{T}_{h}$. 

\bigskip
\begin{definition}
\label{FEM_spaces_mesh}
Let $\mathcal{T}_{h}$ be a tetrahedronal mesh. We define $\mathcal{P}_{\tilde{r}}\Lambda^{3}(\mathcal{T}_{h})$, 
$\mathcal{P}_{\tilde{r}}\Lambda^{2}(\mathcal{T}_{h})$, $\mathcal{P}_{\tilde{r}}^{-}\Lambda^{2}(\mathcal{T}_{h})$, 
$\mathcal{P}_{\tilde{r}}^{-}\Lambda^{1}(\mathcal{T}_{h})$, $\mathcal{P}_{\tilde{r}}\Lambda^{3}(\mathcal{T}_{h};\mathbb{V})$, 
$\mathcal{P}_{\tilde{r}}\Lambda^{2}(\mathcal{T}_{h};\mathbb{V})$,$\mathcal{P}_{\tilde{r}}^{-}\Lambda^{2}
(\mathcal{T}_{h};\mathbb{V})$, and $\mathcal{P}_{\tilde{r}}^{-}\Lambda^{1}(\mathcal{T}_{h};\mathbb{V})$ 
as spaces of piece-wisely smooth functions or vector fields on $\mathcal{T}_{h}$ whose restrictions on $T$ 
are $\mathcal{P}_{\tilde{r}}\Lambda^{3}(T)$, 
$\mathcal{P}_{\tilde{r}}\Lambda^{2}(T)$, $\mathcal{P}_{\tilde{r}}^{-}\Lambda^{2}(T)$, 
$\mathcal{P}_{\tilde{r}}^{-}\Lambda^{1}(T)$, $\mathcal{P}_{\tilde{r}}\Lambda^{3}(T;\mathbb{V})$, 
$\mathcal{P}_{\tilde{r}}\Lambda^{2}(T;\mathbb{V})$,$\mathcal{P}_{\tilde{r}}^{-}\Lambda^{2}
(T;\mathbb{V})$, and $\mathcal{P}_{\tilde{r}}^{-}\Lambda^{1}(T;\mathbb{V})$ respectively, 
for any $T\in\triangle_{3}(\mathcal{T}_{h})$. 
\end{definition}

\begin{remark}
Obviously, we have 
\begin{align*}
& \mathcal{P}_{\tilde{r}}\Lambda^{2}(\mathcal{T}_{h}),\mathcal{P}_{\tilde{r}}^{-}\Lambda^{2}(\mathcal{T}_{h})
\subset H(\text{div},\Omega),\quad
& \mathcal{P}_{\tilde{r}}^{-}\Lambda^{1}(\mathcal{T}_{h})\subset H(\text{curl},\Omega),\\
& \mathcal{P}_{\tilde{r}}\Lambda^{2}(\mathcal{T}_{h};\mathbb{V}), \mathcal{P}_{\tilde{r}}^{-}\Lambda^{2}
(\mathcal{T}_{h};\mathbb{V})\subset H(\text{div},\Omega;\mathbb{M}),\quad
& \mathcal{P}_{\tilde{r}}^{-}\Lambda^{1}(\mathcal{T}_{h};\mathbb{V})\subset H(\text{curl},\Omega;\mathbb{M}).
\end{align*}
Here $\Omega$ is an open subset in $\mathbb{R}^{3}$ with 
$\overline{\Omega}=\cup_{T\in\mathcal{T}_{h}}T$.
The spaces defined in~(\ref{FEM_spaces_mesh}) have been introduced 
in \cite{QD:2009:MMEW} using the language of differential forms.
\end{remark}

\section{Algebraic operators and some auxiliary properties}

In this section, we will introduce two algebraic operators, 
and prove some of their relevant properties.


\begin{definition}
\label{S2}
We introduce a linear map $S_{2}$ defined as follows,
\begin{equation*}
S_{2}U = (u_{23}-u_{32},u_{31}-u_{13},u_{12}-u_{21})^{\top}.
\end{equation*}
Here $U$ is an arbitrary matrix in $\mathbb{R}^{3\times 3}$.
\end{definition}

\begin{remark}
It is easy to check that $S_{2}U=\text{vec}(U^{\top}-U)$.
\end{remark}

\bigskip
\begin{definition}
\label{S1}
We define linear map $S_{1}$ as follows,
\begin{equation*}
S_{1}W = W^{\top}-\text{tr}(W)I.
\end{equation*}
Here $W$ is an arbitrary matrix in $\mathbb{R}^{3\times 3}$.
\end{definition}

\bigskip
\begin{lemma}
\label{S1_invertible}
Operator $S_{1}$ is invertible. And $S_{1}^{-1}W=W^{\top}-\dfrac{1}{2}\text{tr}(W)I$.
\end{lemma}

\begin{lemma}
\label{S1_S2_commuting_diagram}
$\text{div}S_{1}W + S_{2}\text{curl}W=0,\forall W\in H^{1}(\Omega,\mathbb{M}).$ 
Here, $\Omega$ is any open subset in $\mathbb{R}^{3}$.
\end{lemma}

Proofs of Lemma~\ref{S1_invertible} and Lemma~\ref{S1_S2_commuting_diagram} 
are straightforward.

\bigskip
\begin{lemma}
\label{tangential_normal_trace}
Let $T$ be a tetrahedron in $\mathbb{R}^{3}$, and $W\in H^{1}(T,\mathbb{M})$. Let 
$F$ be any face of $T$. If $W\cdot\boldsymbol{t}|_{F}=0$ for all tangential vectors
$\boldsymbol{t}$ on $F$, then $S_{1}W\cdot\boldsymbol{n}|_{F}=0$ where 
$\boldsymbol{n}$ is an unit normal vector on $F$.
\end{lemma}

\begin{proof}
According to definition \ref{S1}, we have
\begin{align*}
S_{1}W\cdot\boldsymbol{n} & =
\left[ 
\begin{array}{ccc}
-w_{22}-w_{33} & w_{21} & w_{31} \\
w_{12} & -w_{11}-w_{33} & w_{32} \\
w_{13} & w_{23} & -w_{11}-w_{22}
\end{array}
\right]\cdot\boldsymbol{n} \\
& = [n_{1}(-w_{22}-w_{33})+n_{2}w_{21}+n_{3}w_{31},
n_{1}w_{12}+n_{2}(-w_{11}-w_{33})+n_{3}w_{32}, \\
&\qquad n_{1}w_{13}+n_{2}w_{23}+n_{3}(-w_{11}-w_{22})]^{\top}.
\end{align*}

In the following, we will show that $n_{1}(-w_{22}-w_{33})+n_{2}w_{21}+n_{3}w_{31}=0$ on $F$. 
The proof of the other two components of $S_{1}W\cdot\boldsymbol{n}$ being zero on $F$
is
similar.

Obviously, $(n_{2},-n_{1},0)^{\top}\cdot\boldsymbol{n}=0$ and $(n_{3},0,-n_{1})^{\top}\cdot
\boldsymbol{n}=0$. This implies that $W\cdot(n_{2},-n_{1},0)^{\top}=0$ and 
$W\cdot(n_{3},0,-n_{1})^{\top}=0$ on $F$. So we have $-n_{1}w_{22}+n_{2}w_{21}=0$ 
and $-n_{1}w_{33}+n_{3}w_{31}=0$ on $F$. This shows that $n_{1}(-w_{22}-w_{33})+n_{2}w_{21}
+n_{3}w_{31}=0$ on $F$. 
\end{proof}

\bigskip
\begin{lemma}
\label{S1_by_part}
For any $W,Q\in\mathbb{M}$, we have $S_{1}W:Q=W:S_{1}Q$.
\end{lemma}

\begin{proof}
\begin{align*}
S_{1}W:Q = &
\left[ 
\begin{array}{ccc}
-w_{22}-w_{33} & w_{21} & w_{31} \\
w_{12} & -w_{11}-w_{33} & w_{32} \\
w_{13} & w_{23} & -w_{11}-w_{22}
\end{array}
\right]:Q \\
= & (-w_{22}-w_{33})q_{11}+w_{21}q_{12}+w_{31}q_{13} \\
  & +w_{12}q_{21}+(-w_{11}-w_{33})q_{22}+w_{32}q_{23} \\
  & +w_{13}q_{31}+w_{23}q_{32}+(-w_{11}-w_{22})q_{33} \\
= & w_{11}(-q_{22}-q_{33})+w_{12}q_{21}+w_{13}q_{31} \\
  & +w_{21}q_{12}+w_{22}(-q_{11}-q_{33})+w_{23}q_{32} \\
  & +w_{31}q_{13}+w_{32}q_{23}+w_{33}(-q_{11}-q_{22}) \\  
= & W:\left[ 
\begin{array}{ccc}
-q_{22}-q_{33} & q_{21} & q_{31} \\
q_{12} & -q_{11}-q_{33} & q_{32} \\
q_{13} & q_{23} & -q_{11}-q_{22}
\end{array}
\right] = W:S_{1}Q.  
\end{align*}
\end{proof}

\begin{lemma}
\label{S1_unisolvent}
Let $T$ be a tetrahedron in $\mathbb{R}^{3}$. We take $W\in\mathring{\mathcal{P}}_{r+2}^{-}
\Lambda^{1}(T;\mathbb{V})$. If 
\begin{equation}
\int_{T}S_{1}W:Q = 0,\quad Q\in \mathcal{P}_{r-1}(T;\mathbb{M}),
\end{equation}
then $W=0$ on $T$.
\end{lemma}

\begin{proof}
According to Lemma \ref{S1_by_part}, we have 
\begin{equation*}
\int_{T}W:S_{1}Q = 0,\quad Q\in \mathcal{P}_{r-1}(T;\mathbb{M}).
\end{equation*}

By the definition of $S_{1}$, it is easy to see that $S_{1}\mathcal{P}_{r-1}(T;\mathbb{M})
\subset\mathcal{P}_{r-1}(T;\mathbb{M})$. According to Lemma \ref{S1_invertible}, we conclude
that $S_{1}\mathcal{P}_{r-1}(T;\mathbb{M})=\mathcal{P}_{r-1}(T;\mathbb{M})$.
According to Lemma $4.11$ in \cite{AFW:2006:ECH}, we have that $W=0$ on $T$.
\end{proof}

\section{Mixed formulation for  the elasticity equations with weakly imposed
symmetry}

We begin by rewriting formulation~(\ref{weak_symm_formula_continuous}) using 
operator $S_{2}$. 
The elasticity problem becomes: Find $(\sigma,u,p)  \in
H(\text{div},\Omega;\mathbb{M}) \times L^{2}(\Omega;\mathbb{V})  
\times L^{2}(\Omega;\mathbb{V})$ such that
\begin{align}
\left\langle A\sigma,\tau\right\rangle +\left\langle \text{div}\tau,u\right\rangle
-\left\langle S_{2}\tau,p\right\rangle  &  =0,\text{ \ }\tau\in H(\text{div},\Omega;\mathbb{M})
,\label{weak_symmetry_S2}\\
\left\langle \text{div}\sigma,v\right\rangle  &  =\left\langle f,v\right\rangle ,\text{
\ }v\in L^{2}(\Omega;\mathbb{V})  ,\nonumber\\
\left\langle S_{2}\sigma,q\right\rangle  &  =0,\text{ \ }q\in L^{2}(\Omega;\mathbb{V}).\nonumber
\end{align}
Here $\left\langle \cdot,\cdot\right\rangle$ is the standard $L^{2}$ inner product on $\Omega$. 
This problem is well-posed in the sense that, for each $f\in L^{2}(\Omega;\mathbb{V})$, 
there exists a unique solution
$(\sigma,u,p)  \in H(\text{div},\Omega;\mathbb{M})  \times L^{2}(\Omega;\mathbb{V})  \times
L^{2}(\Omega;\mathbb{V})$, and the solution operator is a bounded operator
\[
L^{2}(\Omega;\mathbb{V})\longrightarrow
H(\text{div},\Omega;\mathbb{M})\times L^{2}(\Omega;\mathbb{V})\times L^{2}(\Omega;\mathbb{V}).
\]
See \cite{AFW:2006:ECH} and \cite{Falk:2008:FME} for the proof.

Next, we consider a finite element discretization of
(\ref{weak_symmetry_S2}). For this, we choose families of
finite-dimensional subspaces
\[
\Lambda_{h}^{2}(\mathbb{M})  \subset H(\text{div},\Omega;\mathbb{M}), 
\Lambda_{h}^{3}(\mathbb{V})  \subset
L^{2}(\Omega;\mathbb{V})  ,\overline{\Lambda}_{h}^{3}(\mathbb{V})
\subset L^{2}(\Omega;\mathbb{V}),
\]
indexed by $h$, and seek the discrete solution $(\sigma_{h},u_{h},p_{h})
\in\Lambda_{h}^{2}(\mathbb{M})\times\Lambda_{h}^{3}(\mathbb{V})\times
\overline{\Lambda}_{h}^{3}(\mathbb{V})$ such that
\begin{align}
\left\langle A\sigma_{h},\tau\right\rangle +\left\langle \text{div}\tau,u_{h}
\right\rangle -\left\langle S_{2}\tau,p_{h}\right\rangle  &  =0,\text{ \ }\tau
\in\Lambda_{h}^{2}(\mathbb{M})
,\label{weak_symmetry_discrete}\\
\left\langle \text{div}\sigma_{h},v\right\rangle  &  =\left\langle f,v\right\rangle
,\text{ \ }v\in\Lambda_{h}^{3}(\mathbb{V})  ,\nonumber\\
\left\langle S_{2}\sigma_{h},q\right\rangle & =0,\text{ \ }q\in
\overline{\Lambda}_{h}^{3}(\mathbb{V})  .\nonumber
\end{align}
The stability of (\ref{weak_symmetry_discrete}) will be
ensured by the Brezzi stability conditions:
\begin{gather}
\text{(S1) }\Vert\tau\Vert_{H(\text{div},\Omega;\mathbb{M})}^{2}\leq c_{1}
\langle A\tau,\tau\rangle  \text{ whenever }\tau\in\Lambda_{h}^{2}
(\mathbb{M})  \text{ satisfies }\langle\text{div}\tau,v\rangle
=0\label{S1_condition}\\
\forall v\in\Lambda_{h}^{3}(\mathbb{V})  \text{ and }\langle S_{2}\tau,q\rangle =0
\text{ }\forall q\in\overline{\Lambda}_{h}^{3}(\mathbb{V}),\nonumber
\end{gather}
\begin{gather}
\text{(S2) for all nonzero }(v,q)  \in\Lambda_{h}^{3}
(\mathbb{V})\times\overline{\Lambda}_{h}^{3}(\mathbb{V})  \text{,
there exists nonzero}\label{S2_condition}\\
\tau\in\Lambda_{h}^{2}(\mathbb{M})  \text{ with }\langle
\text{div}\tau,v\rangle -\langle S_{2}\tau,q\rangle \geq 
c_{2}\Vert\tau\Vert_{H(\text{div},\Omega;\mathbb{M})}
(\Vert v\Vert_{L^{2}(\Omega;\mathbb{V})} +\Vert q\Vert_{L^{2}(\Omega;\mathbb{V})} )  ,\nonumber
\end{gather}
where constants $c_{1}$ and $c_{2}$ are independent of $h$.

For meshes of arbitrary but uniform order, conditions (\ref{S1_condition})
and (\ref{S2_condition}) have been proved in \cite{AFW:2006:ECH} and
\cite{Falk:2008:FME}. In what follows, we will demonstrate that they are also
satisfied for meshes with elements of variable (but limited) order.
In this paper, we define $\Lambda_{h}^{2}(\mathbb{M})=\mathcal{P}_{\tilde{r}+1}\Lambda^{2}
(\mathcal{T}_{h};\mathbb{V})$, and $\Lambda_{h}^{3}(\mathbb{V})=\overline{\Lambda}_{h}^{3}(\mathbb{V})=\mathcal{P}_{\tilde{r}}\Lambda^{3}(\mathcal{T}_{h};\mathbb{V})$.
We assume that there is $r_{\max}\in\mathbb{N}$ such that for 
any $h>0$ and $f\in\Delta (\mathcal{T}_{h})$,$\tilde{r}(f)\leq r_{\max}$.

\section{Preliminaries for the proof of stability}

From now on, we assume that $\Omega$ is a bounded polyhedral
domain in $\mathbb{R}^{3}$. We also use the standard assumptions for shape
regular meshes, which means that the ratio between outer diameter 
and inner diameter of any tetrahedron in any mesh has an uniform upper bound.

In the proof of stability, the following three commuting diagrams are essential.
\begin{equation}
\begin{array}{ccc}
H^{1}(\Omega;\mathbb{M}) & \overset{\text{div}}{\longrightarrow} & L^{2}(\Omega;\mathbb{V}) \\
\Pi_{\tilde{r},h}^{2}\downarrow & & \Pi_{\tilde{r},h}^{3}\downarrow \\
\mathcal{P}_{\tilde{r}+1}\Lambda^{2}(\mathcal{T}_{h};\mathbb{V}) & 
\overset{\text{div}}{\longrightarrow} & \mathcal{P}_{\tilde{r}}\Lambda^{3}(\mathcal{T}_{h};\mathbb{V}) 
\end{array}
\label{commuting_diagram1}
\end{equation}

\begin{equation}
\begin{array}{ccc}
H^{1}(\Omega;\mathbb{M}) & \overset{\text{div}}{\longrightarrow} & L^{2}(\Omega;\mathbb{V}) \\
\Pi_{\tilde{r},h}^{2,-}\downarrow & & \Pi_{\tilde{r},h}^{3}\downarrow \\
\mathcal{P}_{\tilde{r}+1}^{-}\Lambda^{2}(\mathcal{T}_{h};\mathbb{V}) & 
\overset{\Pi_{\tilde{r},h}^{3}\circ\text{div}}{\longrightarrow} & \mathcal{P}_{\tilde{r}}
\Lambda^{3}(\mathcal{T}_{h};\mathbb{V}) 
\end{array}
\label{commuting_diagram2}
\end{equation}

\begin{equation}
\begin{array}{ccc}
H^{1}(\Omega;\mathbb{M}) & \overset{S_{1}}{\longrightarrow} & H^{1}(\Omega;\mathbb{M}) \\
\overline{\Pi}_{\tilde{r},h}^{1,-}\downarrow & & \Pi_{\tilde{r},h}^{2,-}\downarrow \\
\mathcal{P}_{\tilde{r}+2}^{-}\Lambda^{1}(\mathcal{T}_{h};\mathbb{V}) & \overset{\Pi_{\tilde{r},h}^{2,-}
\circ S_{1}}{\longrightarrow} & \mathcal{P}_{\tilde{r}+1}^{-}\Lambda^{2}(\mathcal{T}_{h};\mathbb{V}) 
\end{array}
\label{commuting_diagram3}
\end{equation}

Here $\Pi_{\tilde{r},h}^{3}$ is the $L^{2}$ orthogonal projection operator onto $\mathcal{P}_{\tilde{r}}
\Lambda^{3}(\mathcal{T}_{h};\mathbb{V})$. $\Pi_{\tilde{r},h}^{2}$, $\Pi_{\tilde{r},h}^{2,-}$, 
and $\overline{\Pi}_{\tilde{r},h}^{1,-}$ are projection operators 
into $\mathcal{P}_{\tilde{r}+1}\Lambda^{2}(\mathcal{T}_{h};\mathbb{V})$, 
$\mathcal{P}_{\tilde{r}+1}^{-}\Lambda^{2}(\mathcal{T}_{h};\mathbb{V})$, and 
$\mathcal{P}_{\tilde{r}+2}^{-}\Lambda^{1}(\mathcal{T}_{h};\mathbb{V})$ respectively.

In \cite{AFW:2006:ECH,AFW:2007:MMEW}, the canonical projection operators introduced 
by Arnold, Falk and Winther can make (\ref{commuting_diagram1},\ref{commuting_diagram2},
\ref{commuting_diagram3}) commute for meshes with uniform order. But for meshes with 
variable order, the natural generalization of the projection operators fails to
make commute both~(\ref{commuting_diagram1}) and~(\ref{commuting_diagram2}), see 
a counter-example 
presented in the appendix of \cite{QD:2009:MMEW}. To overcome the difficulty, 
we recalled 
Projection Based (PB) interpolation operators from \cite{QD:2009:MMEW}.
According to Lemma $19$ and 
Lemma $20$ in \cite{QD:2009:MMEW}, there exists projection based interpolation operator 
$\Pi_{\tilde{r},h}^{2}$, which satisfies the following properties. 

\begin{equation}
\text{div}\Pi_{\tilde{r},h}^{2}\tau=\Pi_{\tilde{r},h}^{3}\text{div}\tau,\quad 
\tau\in H^{1}(\Omega;\mathbb{M}).
\label{cd1}
\end{equation}

\begin{equation}
\Vert\Pi_{\tilde{r},h}^{2}\tau\Vert\leq C\Vert\tau\Vert_{H^{1}(\Omega;\mathbb{M})}, 
\quad\tau\in H^{1}(\Omega;\mathbb{M}).
\label{cd1bound}
\end{equation} 
Here $C$ is independent of $\tau$ and $h$. Though Lemma $20$ in 
\cite{QD:2009:MMEW} has been proved for quasi-uniform meshes only, 
it is straightforward to 
extend it to get (\ref{cd1bound}) for shape regular meshes as well. 
Please refer to \cite{QD:2009:MMEW} 
for the details on the PB interpolation operators.

With the PB operators in place, the difficulty shifted to 
to defining 
a special projection operator $\overline{\Pi}_{\tilde{r},h}^{1,-}$, 
denoted by $W_{h}$ in \cite{QD:2009:MMEW}, 
that makes now~(\ref{commuting_diagram3}) commute. 
The commutativity property followed directly from the construction of $W_{h}$
but proving that it is well-defined, turned out to be difficult.
We managed to show only that
$\overline{\Pi}_{\tilde{r},h}^{1,-}$ is well-defined 
for $0\leq\tilde{r}\leq 3$ with $n=2$. In the following, we will use a different
reasoning to demonstrate that 
{\em there exist} projection operators $\Pi_{\tilde{r},h}^{2,-}$ and 
$\overline{\Pi}_{\tilde{r},h}^{1,-}$, both well-defined, that make 
(\ref{commuting_diagram2}),(\ref{commuting_diagram3}) commute for arbitrary 3D 
meshes of arbitrary order. Note that the operators will not be constructed
explicitly.

\subsection{Projection operators on a reference tetrahedron}

Let $\hat{T}$ be a fixed tetrahedron in $\mathbb{R}^{3}$. We are going to 
design projection operators $\Pi_{\tilde{r},\hat{T}}^{2,-}$ and 
$\Pi_{\tilde{r},\hat{T}}^{1,-}$ into $\mathcal{P}_{\tilde{r}+1}^{-}\Lambda^{2}
(\hat{T};\mathbb{V})$, and $\mathcal{P}_{\tilde{r}+2}^{-}\Lambda^{1}(\hat{T};\mathbb{V})$ 
respectively.

\begin{definition}
We take $\tilde{r}$ to be a mapping from $\triangle (\hat{T})$ to $\mathbb{Z}_{+}$ 
such that if $\hat{e},\hat{f}\in\triangle (\hat{T})$ and $\hat{e}\subset\hat{f}$, 
then $\tilde{r}(\hat{e})\leq\tilde{r}(\hat{f})$. We put $k=\dim
\text{curl}_{\hat{\boldsymbol{x}}}\mathring{\mathcal{P}}_{\tilde{r}(\hat{T})+1}\Lambda^{1}
(\hat{T};\mathbb{V})$.

We define $\{\hat{\boldsymbol{f}}_{\tilde{r},1},\cdots,\hat{\boldsymbol{f}}_{\tilde{r},k}\}$ 
as a basis of $\text{curl}_{\hat{\boldsymbol{x}}}\mathring{\mathcal{P}}_{\tilde{r}(\hat{T})+1}
\Lambda^{1}(\hat{T};\mathbb{V})$. We define $\{\hat{\boldsymbol{g}}_{\tilde{r},1},\cdots,
\hat{\boldsymbol{g}}_{\tilde{r},k}\}$ as a linearly independent subset of $\mathcal{P}_{\tilde{r}
(\hat{T})-1}(\hat{T};\mathbb{M})$ such that $\{\hat{\boldsymbol{g}}_{\tilde{r},1},\cdots,
\hat{\boldsymbol{g}}_{\tilde{r},k}\}\oplus\nabla_{\hat{\boldsymbol{x}}}\mathcal{P}_{\tilde{r}}
\Lambda^{3}(\hat{T};\mathbb{V})=\mathcal{P}_{\tilde{r}(\hat{T})-1}(\hat{T};\mathbb{M})$. 
We define $\hat{\boldsymbol{h}}_{\tilde{r},i}(\hat{\boldsymbol{x}},t)=(1-t)\hat{\boldsymbol{f}}
_{\tilde{r}}(\hat{\boldsymbol{x}})+t\hat{\boldsymbol{g}}_{\tilde{r}}(\hat{\boldsymbol{x}})$ for 
any $1\leq i\leq k$.
\end{definition}

\begin{remark}
It is easy to check that $k = \dim\mathcal{P}_{\tilde{r}(\hat{T})-1}(\hat{T};\mathbb{M})-
\dim\nabla_{\hat{\boldsymbol{x}}}\mathcal{P}_{\tilde{r}}\Lambda^{3}(\hat{T};\mathbb{V})$. 
Take $r = \tilde{r}(\hat{T})$. Then
$\dim\text{curl}_{\hat{\boldsymbol{x}}}\mathring{\mathcal{P}}_{r+1}\Lambda^{1}
(\hat{T};\mathbb{V})=3(\dim\mathring{\mathcal{P}}_{r+1}\Lambda^{1}(\hat{T})-
\dim\mathring{\mathcal{P}}_{r+2}\Lambda^{0}(\hat{T}))
=\dfrac{1}{2}(2r+5)r(r-1)$. 
$\dim\mathcal{P}_{r-1}(\hat{T};\mathbb{M})-\dim\nabla_{\hat{\boldsymbol{x}}}\mathcal{P}_{r}
\Lambda^{3}(\hat{T};\mathbb{V})=3(\dim\mathcal{P}_{r-1}(\hat{T};\mathbb{V})-\dim\mathcal{P}_{r}(\hat{T})
/\mathbb{R})=\dfrac{1}{2}(2r+5)r(r-1)$.
For the dimensions of finite element spaces mentioned above, please refer to formula ($3.1$) 
in \cite{AFW:2006:ECH} 
and page $51$ in \cite{AFW:2006:ECH}.
\end{remark}

\bigskip
\begin{definition}
\label{op2_t}
For any $t\in [0,1]$, we define the linear operator $\Pi_{\tilde{r},\hat{T},t}^{2,-}$ 
mapping $H^{1}(\hat{T};\mathbb{M})$ 
onto $\mathcal{P}_{\tilde{r}+1}^{-}\Lambda^{2}(\hat{T};\mathbb{V})$ 
by the following conditions. 
\begin{equation}
\int_{\hat{T}}\text{div}_{\hat{\boldsymbol{x}}}(\Pi_{\tilde{r},\hat{T},t}^{2,-}\hat{U}-\hat{U})\cdot\hat{\eta}d\hat{\boldsymbol{x}}=0,
\quad\hat{\eta}\in\mathcal{P}_{\tilde{r}(\hat{T})}(\hat{T};\mathbb{V})/\mathbb{R}.
\label{op2_t_div}
\end{equation}
\begin{equation}
\int_{\hat{T}}(\Pi_{\tilde{r},\hat{T},t}^{2,-}\hat{U}-\hat{U}):\hat{\boldsymbol{h}}_{\tilde{r},i}(\hat{\boldsymbol{x}},t)d\hat{\boldsymbol{x}}=0,\quad 1\leq i\leq k.
\label{op2_t_aux}
\end{equation}
\begin{equation}
\int_{\hat{F}}[(\Pi_{\tilde{r},\hat{T},t}^{2,-}\hat{U}-\hat{U})\cdot\hat{\boldsymbol{n}}]\cdot
\hat{\mu}d\hat{s}=0, \quad\hat{F}\in\triangle_{2}(\hat{T}),\quad
\hat{\mu}\in\mathcal{P}_{\tilde{r}(\hat{F})}(\hat{F};\mathbb{V}). 
\label{op2_t_face}
\end{equation}
\end{definition}

\begin{definition}
\label{op1_t}
For any $t\in [0,1]$, we define the linear operator $\Pi_{\tilde{r},\hat{T},t}^{1,-}$ 
mapping $H^{1}(\hat{T};\mathbb{M})$ 
into $\mathcal{P}_{\tilde{r}+2}^{-}\Lambda^{1}(\hat{T};\mathbb{V})$ by the 
following conditions. 
\begin{equation}
\int_{\hat{T}}\text{div}_{\hat{\boldsymbol{x}}}S_{1}(\Pi_{\tilde{r},\hat{T},t}^{1,-}\hat{W}-\hat{W})\cdot\hat{\eta}d\hat{\boldsymbol{x}}=0, \quad
\hat{\eta}\in\mathcal{P}_{\tilde{r}(\hat{T})}(\hat{T};\mathbb{V})/\mathbb{R}.
\label{op1_t_div}
\end{equation}
\begin{equation}
\int_{\hat{T}}S_{1}(\Pi_{\tilde{r},\hat{T},t}^{1,-}\hat{W}-\hat{W}):\hat{\boldsymbol{h}}_{\tilde{r},i}(\hat{\boldsymbol{x}},t)d\hat{\boldsymbol{x}}=0, \quad 1\leq i\leq k.
\label{op1_t_aux}
\end{equation}
\begin{equation}
\int_{\hat{F}}[(\Pi_{\tilde{r},\hat{T},t}^{1,-}\hat{W}-\hat{W})\cdot\hat{\boldsymbol{t}}]\cdot
\hat{\mu}d\hat{s}=0, \quad\hat{F}\in\triangle_{2}(\hat{T}),\quad
\hat{\mu}\in\mathcal{P}_{\tilde{r}(\hat{F})}(\hat{F};\mathbb{V}). 
\label{op1_t_face}
\end{equation}
\begin{equation}
\Pi_{\tilde{r},\hat{T},t}^{1,-}\hat{W}\cdot\hat{\boldsymbol{t}}|_{\hat{e}}=0, \quad\hat{e}
\in\triangle_{1}(\hat{T}).
\label{op1_t_edge}
\end{equation}
\end{definition}

In (\ref{op2_t_face}), $\hat{\boldsymbol{n}}$ is a unit normal vector on $\hat{F}$. 
In (\ref{op1_t_face}), $\hat{\boldsymbol{t}}$ is any tangential vector on $\hat{F}$. 
Notice that the dimension of tangential vector space on $\hat{F}$ is two. 
In (\ref{op1_t_edge}), 
$\hat{\boldsymbol{t}}$ is a tangential vector along $\hat{e}$.

\bigskip
\begin{lemma}
\label{op2_t_well_defined}
For any $\tilde{r}(\hat{T})\in\mathbb{Z}_{+}$, 
operator $\Pi_{\tilde{r},\hat{T},t}^{2,-}$ is a linear 
projection, and a well-defined operator for all but finitely many values of $t\in [0,1]$.
\end{lemma}

\begin{proof}
It is easy to see that the conditions (\ref{op2_t_div},\ref{op2_t_aux},\ref{op2_t_face}) 
are well-defined for 
any $\hat{U}\in H^{1}(\hat{T};\mathbb{M})$. Obviously, if $\Pi_{\tilde{r},\hat{T},t}^{2,-}$ 
is well-defined, then it is linear and a projection. It is sufficient to show that 
for any $\hat{U}\in
\mathcal{P}_{\tilde{r}+1}^{-}\Lambda^{2}(\hat{T};\mathbb{V})$, 
$\hat{U}=0$ if $\Pi_{\tilde{r},\hat{T},t}^{2,-}
\hat{U}=0$.

By Theorem $4.12$ in \cite{AFW:2006:ECH}, 
$\hat{U}\in\mathring{\mathcal{P}}_{\tilde{r}(\hat{T})+1}
^{-}\Lambda^{2}(\hat{T};\mathbb{V})$ because 
\begin{equation*}
\int_{\hat{F}}[\hat{U}\cdot\hat{\boldsymbol{n}}]\cdot
\hat{\mu}d\hat{s}=0, \quad \hat{F}\in\triangle_{2}(\hat{T}),\quad
\hat{\mu}\in\mathcal{P}_{\tilde{r}(\hat{F})}(\hat{F};\mathbb{V}). 
\end{equation*}
So it is sufficient to show that, for any $\hat{U}\in\mathring{\mathcal{P}}_{\tilde{r}(\hat{T})+1}
^{-}\Lambda^{2}(\hat{T};\mathbb{V})$, $\hat{U}=0$, provided,
\begin{equation}
\int_{\hat{T}}\text{div}_{\hat{\boldsymbol{x}}}\hat{U}\cdot\hat{\eta}d\hat{\boldsymbol{x}}=0,
\quad\hat{\eta}\in\mathcal{P}_{\tilde{r}(\hat{T})}(\hat{T};\mathbb{V})/\mathbb{R};
\label{op2_t_div_proof}
\end{equation}
\begin{equation}
\int_{\hat{T}}\hat{U}:\hat{\boldsymbol{h}}_{\tilde{r},i}(\hat{\boldsymbol{x}},t)d
\hat{\boldsymbol{x}}=0,\quad 1\leq i\leq k.
\label{op2_t_aux_proof}
\end{equation}
Since $\hat{U}\in\mathring{\mathcal{P}}_{\tilde{r}(\hat{T})+1}
^{-}\Lambda^{2}(\hat{T};\mathbb{V})$, (\ref{op2_t_div_proof}) can be integrated by 
parts to yield 
\begin{equation}
\int_{\hat{T}}\hat{U}:\nabla_{\hat{\boldsymbol{x}}}\hat{\eta}d\hat{\boldsymbol{x}}=0,
\quad\hat{\eta}\in\mathcal{P}_{\tilde{r}(\hat{T})}(\hat{T};\mathbb{V}).
\label{op2_t_div_proof2}
\end{equation}

According to the definition of $\hat{\boldsymbol{h}}_{\tilde{r},i}(\hat{\boldsymbol{x}},t)$ 
and the fact that $\mathcal{P}_{\tilde{r}}\Lambda^{3}(\hat{T};\mathbb{V})=
\mathcal{P}_{\tilde{r}(\hat{T})}(\hat{T};\mathbb{V})$, the assertion is true for $t=0$.
Indeed, when $t=0$, conditions (\ref{op2_t_div_proof2},\ref{op2_t_aux_proof}) 
can be rewritten as 
\begin{equation*}
\int_{\hat{T}}\hat{U}:\hat{Q}d\hat{\boldsymbol{x}}=0,
\quad\hat{Q}\in\mathcal{P}_{\tilde{r}(\hat{T})-1}(\hat{T};\mathbb{M}).
\end{equation*}
By Lemma $4.11$ in \cite{AFW:2006:ECH}, we have $\hat{U}=0$. This implies that 
$\Pi_{\tilde{r},\hat{T},t}^{2,-}$ is well-defined for $t=0$.

We denote by $C(t)$ the matrix associated with the left hand side of conditions 
(\ref{op2_t_div},\ref{op2_t_aux},\ref{op2_t_face}). Then $\Pi_{\tilde{r},\hat{T},t}^{2,-}$ 
is well-defined if and only if $C(t)$ is a square nonsingular matrix. We have already known 
that $C(0)$ is a square nonsingular matrix. So $C(t)$ is a square matrix for any $t\in [0,1]$.
Notice that $\det(C(t))$ is a polynomial of a single variable $t$. Since $\det(C(0))\neq 0$, 
then there are at most finitely many $t\in [0,1]$ which make $\det(C(t))=0$. 
This implies that 
$\Pi_{\tilde{r},\hat{T},t}^{2,-}$ is well-defined for all but finitely many 
values of $t\in [0,1]$. 
\end{proof}

\bigskip
\begin{lemma}
\label{op1_t_well_defined}
For any $\tilde{r}(\hat{T})\in\mathbb{Z}_{+}$, operator $\Pi_{\tilde{r},\hat{T},t}^{1,-}$ 
is a well-defined, linear 
projection operator for all but finitely many values of $t\in [0,1]$.
\end{lemma}

\begin{proof}
It is easy to see that the conditions~(\ref{op1_t_div},\ref{op1_t_aux},\ref{op1_t_face},
\ref{op1_t_edge}) 
are well-defined for any $\hat{W}\in H^{1}(\hat{T};\mathbb{M})$. Obviously, 
if $\Pi_{\tilde{r},\hat{T},t}^{1,-}$ 
is well-defined, then it is linear and a projection. It is sufficient to show 
that for any $\hat{W}\in
\mathcal{P}_{\tilde{r}+2}^{-}\Lambda^{1}(\hat{T};\mathbb{V})$ with $\hat{W}\cdot\hat{\boldsymbol{t}}|_{\hat{e}}=0$, 
then $\hat{W}=0$ if $\Pi_{\tilde{r},\hat{T},t}^{1,-}\hat{W}=0$. 
Here $\hat{e}$ is any edge of $\hat{T}$, and 
$\hat{\boldsymbol{t}}$ is a tangential vector along $\hat{e}$. 

By Theorem $4.12$ in \cite{AFW:2006:ECH}, 
$\hat{W}\in\mathring{\mathcal{P}}_{\tilde{r}(\hat{T})+2}
^{-}\Lambda^{1}(\hat{T};\mathbb{V})$ because  
\begin{equation*}
\int_{\hat{F}}[\hat{W}\cdot\hat{\boldsymbol{t}}]\cdot
\hat{\mu}d\hat{s}=0, \quad\hat{F}\in\triangle_{2}(\hat{T}),\quad
\hat{\mu}\in\mathcal{P}_{\tilde{r}(\hat{F})}(\hat{F};\mathbb{V}); 
\end{equation*}
\begin{equation*}
\hat{W}\cdot\hat{\boldsymbol{t}}|_{\hat{e}}=0,\quad\hat{e}
\in\triangle_{1}(\hat{T}).
\end{equation*}
So it is sufficient to show that, 
for any $\hat{W}\in\mathring{\mathcal{P}}_{\tilde{r}(\hat{T})+2}
^{-}\Lambda^{1}(\hat{T};\mathbb{V})$, $\hat{W}=0$, provided,
\begin{equation}
\int_{\hat{T}}\text{div}_{\hat{\boldsymbol{x}}}S_{1}\hat{W}\cdot\hat{\eta}d\hat{\boldsymbol{x}}=0,
\quad\hat{\eta}\in\mathcal{P}_{\tilde{r}(\hat{T})}(\hat{T};\mathbb{V})/\mathbb{R};
\label{op1_t_div_proof}
\end{equation}
\begin{equation}
\int_{\hat{T}}S_{1}\hat{W}:\hat{\boldsymbol{h}}_{\tilde{r},i}(\hat{\boldsymbol{x}},t)
d\hat{\boldsymbol{x}}=0,\quad 1\leq i\leq k.
\label{op1_t_aux_proof}
\end{equation}

Notice that $\hat{W}\in\mathring{\mathcal{P}}_{\tilde{r}(\hat{T})+2}
^{-}\Lambda^{1}(\hat{T};\mathbb{V})$. By Lemma \ref{tangential_normal_trace}, 
we have that $S_{1}\hat{W}\cdot\hat{\boldsymbol{n}}|_{\hat{F}}=0$ 
for any $\hat{F}\in\triangle_{2}(\hat{T})$. So we can integrate thus 
(\ref{op1_t_div_proof}) by parts without obtaining any boundary term. 
Condition (\ref{op1_t_div_proof}) 
can be rewritten as follows.
\begin{equation}
\int_{\hat{T}}S_{1}\hat{W}:\nabla_{\hat{\boldsymbol{x}}}\hat{\eta}d\hat{\boldsymbol{x}}=0,
\quad\hat{\eta}\in\mathcal{P}_{\tilde{r}(\hat{T})}(\hat{T};\mathbb{V}).
\label{op1_t_div_proof2}
\end{equation}
By the definition of $\hat{\boldsymbol{h}}_{\tilde{r},i}(\hat{\boldsymbol{x}},t)$, for $t=1$,
conditions (\ref{op1_t_div_proof2}) and (\ref{op1_t_aux_proof}) can be rewritten as
\begin{equation*}
\int_{\hat{T}}S_{1}\hat{W}:\hat{Q}d\hat{\boldsymbol{x}}=0,
\quad\hat{Q}\in\mathcal{P}_{\tilde{r}
(\hat{T})-1}(\hat{T};\mathbb{M}).
\end{equation*}
Lemma~\ref{S1_unisolvent} implies then that $\hat{W}=0$. 
This shows that $\Pi_{\tilde{r},\hat{T},t}^{1,-}$ 
is well-defined for $t=1$. 

We denote by $C(t)$ the matrix associated with the left hand side of conditions 
(\ref{op1_t_div},\ref{op1_t_aux},\ref{op1_t_face},\ref{op1_t_edge}). 
Then $\Pi_{\tilde{r},\hat{T},t}^{1,-}$ 
is well-defined if and only if $C(t)$ is a square non-singular matrix. 
We have already known 
that $C(1)$ is a square non-singular matrix. So $C(t)$ is a square matrix 
for any $t\in [0,1]$.
Notice that $\det(C(t))$ is a polynomial of a single variable $t$. 
Since $\det(C(1))\neq 0$, 
then there are at most finitely many $t\in [0,1]$ 
which make $\det(C(t))=0$. This implies that 
$\Pi_{\tilde{r},\hat{T},t}^{1,-}$ is well-defined for all but 
finitely many values of $t\in [0,1]$. 
\end{proof}

\bigskip
According to Lemma~\ref{op2_t_well_defined} and Lemma~\ref{op1_t_well_defined}, 
we can choose 
$t_{r}\in [0,1]$ for any $\tilde{r}(\hat{T})$ 
such that both $\Pi_{\tilde{r},\hat{T},t_{r}}^{2,-}$ 
and $\Pi_{\tilde{r},\hat{T},t_{r}}^{1,-}$ are well-defined. 
Here $t_{r}$ depends only on $\tilde{r}(\hat{T})$.

\begin{definition}
\label{operators_reference}
We define operators $\Pi_{\tilde{r},\hat{T}}^{2,-}:=\Pi_{\tilde{r},\hat{T},t_{r}}^{2,-}$ 
and $\Pi_{\tilde{r},\hat{T}}^{1,-}:=\Pi_{\tilde{r},\hat{T},t_{r}}^{1,-}$.
\end{definition}

\subsection{Projection operators on a physical tetrahedron}

Let $T$ be an arbitrary tetrahedron in $\mathbb{R}^{3}$. Then there exists an affine mapping 
from the reference tetrahedron $\hat{T}$ to $T$, defined by 
\begin{equation}
\boldsymbol{x}=A\hat{\boldsymbol{x}}+\boldsymbol{b}.
\label{affine_transform}
\end{equation}
Here $A$ is a $3\times 3$ real non-singular matrix, and $\boldsymbol{b}$ 
is a vector in $\mathbb{R}^{3}$.
In the following, we always relate $\boldsymbol{x}$ 
and $\hat{\boldsymbol{x}}$ by (\ref{affine_transform}).
We take $\tilde{r}$ to be a mapping from $\triangle (T)$ to $\mathbb{Z}_{+}$ 
such that if $e,f\in\triangle (T)$ and $e\subset f$, 
then $\tilde{r}(e)\leq\tilde{r}(f)$. In the following, 
we denote by $\hat{\boldsymbol{x}}(\boldsymbol{x})$ 
the inverse of the affine mapping described above.

\bigskip
\begin{definition}
\label{op2_phy}
We define the linear operator $\Pi_{\tilde{r},T}^{2,-}$ mapping $H^{1}(T;\mathbb{M})$ onto 
$\mathcal{P}_{\tilde{r}+1}^{-}\Lambda^{2}(T;\mathbb{V})$ by the following conditions. 
\begin{equation}
\int_{T}\text{div}(\Pi_{\tilde{r},T}^{2,-}U-U)\cdot\eta d\boldsymbol{x}=0,\quad
\eta\in\mathcal{P}_{\tilde{r}(T)}(T;\mathbb{V})/\mathbb{R}.
\label{op2_div_phy}
\end{equation}
\begin{equation}
\int_{T}(\Pi_{\tilde{r},T}^{2,-}U-U)(\boldsymbol{x}):[A\hat{\boldsymbol{h}}_{\tilde{r},i}(\hat{\boldsymbol{x}}(\boldsymbol{x}),t_{r})A^{-1}]d\boldsymbol{x}=0,\quad 1\leq i\leq k.
\label{op2_aux_phy}
\end{equation}
\begin{equation}
\int_{F}[(\Pi_{\tilde{r},T}^{2,-}U-U)\cdot\boldsymbol{n}]\cdot
\mu ds=0, \quad F\in\triangle_{2}(T),\quad
\mu\in\mathcal{P}_{\tilde{r}(F)}(F;\mathbb{V}). 
\label{op2_face_phy}
\end{equation}
\end{definition}

\begin{definition}
\label{op1_phy}
We define the linear operator $\Pi_{\tilde{r},T}^{1,-}$ mapping $H^{1}(T;\mathbb{M})$ 
into $\mathcal{P}_{\tilde{r}+2}^{-}\Lambda^{1}(T;\mathbb{V})$ by the following conditions. 
\begin{equation}
\int_{T}\text{div}S_{1}(\Pi_{\tilde{r},T}^{1,-}W-W)\cdot\eta d\boldsymbol{x}=0, \quad
\eta\in\mathcal{P}_{\tilde{r}(T)}(T;\mathbb{V})/\mathbb{R}.
\label{op1_div_phy}
\end{equation}
\begin{equation}
\int_{T}S_{1}(\Pi_{\tilde{r},T}^{1,-}W-W):[A\hat{\boldsymbol{h}}_{\tilde{r},i}(\hat{\boldsymbol{x}}(\boldsymbol{x}),t_{r})A^{-1}]d\boldsymbol{x}=0,\quad 1\leq i\leq k.
\label{op1_aux_phy}
\end{equation}
\begin{equation}
\int_{F}[(\Pi_{\tilde{r},T}^{1,-}W-W)\cdot\boldsymbol{t}]\cdot
\mu ds=0, \quad F\in\triangle_{2}(T),\quad
\mu\in\mathcal{P}_{\tilde{r}(F)}(F;\mathbb{V}). 
\label{op1_face_phy}
\end{equation}
\begin{equation}
\Pi_{\tilde{r},T}^{1,-}W\cdot\boldsymbol{t}|_{e}=0, \quad e
\in\triangle_{1}(T).
\label{op1_edge_phy}
\end{equation}
\end{definition}

We want to ``pull back" $\tilde{r}$ from $\triangle (T)$ to $\triangle (\hat{T})$.
We put $\tilde{r}(\hat{e})=\tilde{r}(e)$ for any $\hat{e}\in\triangle (\hat{T})$. 
Here $e:=A\hat{e}+\boldsymbol{b}$. Then we have the following lemma.

\begin{lemma}
\label{op2_op1_well_defined}
For any $U,W\in H^{1}(T;\mathbb{M})$, we define $\hat{U},
\hat{W}\in H^{1}(\hat{T};\mathbb{M})$ by 
\begin{equation*}
U(\boldsymbol{x})=A^{-\top}\hat{U}(\hat{\boldsymbol{x}})A^{\top},\quad 
W(\boldsymbol{x})=A\hat{W}(\hat{\boldsymbol{x}})A^{-1}.
\end{equation*}

Then we have 
\begin{equation}
\Pi_{\tilde{r},T}^{2,-}U(\boldsymbol{x})=A^{-\top}\Pi_{\tilde{r},\hat{T}}^{2,-}\hat{U}
(\hat{\boldsymbol{x}})A^{\top},\quad 
\Pi_{\tilde{r},T}^{1,-}W(\boldsymbol{x})=A\Pi_{\tilde{r},\hat{T}}^{1,-}\hat{W}
(\hat{\boldsymbol{x}})A^{-1},\quad\boldsymbol{x}\in T.
\label{pull_back_op2_op1}
\end{equation}

So operators $\Pi_{\tilde{r},T}^{2,-}$ and $\Pi_{\tilde{r},T}^{1,-}$ are well-defined.
\end{lemma}

\begin{proof}
For the result of $\Pi_{\tilde{r},T}^{2,-}$, the proof is straightforward. For the result 
of $\Pi_{\tilde{r},T}^{1,-}$, we need utilize the definition of $S_{1}$. Notice that 
\begin{equation}
S_{1}W(\boldsymbol{x})=W(\boldsymbol{x})^{\top}-\text{tr}(W(\boldsymbol{x}))I
=A^{-\top}[\hat{W}(\hat{\boldsymbol{x}})^{\top}-\text{tr}(\hat{W}(\hat{\boldsymbol{x}}))I]A^{\top}
=A^{-\top}S_{1}\hat{W}(\hat{\boldsymbol{x}})A^{\top}.
\label{S1_transform}
\end{equation}
Using (\ref{S1_transform}), it is now straightforward to 
prove the result for $\Pi_{\tilde{r},T}^{1,-}$.
\end{proof}

\bigskip
\begin{lemma}
\label{cd_div_phy}
For any $U\in H^{1}(T;\mathbb{M})$, we have 
\begin{equation*}
\Pi_{\tilde{r},T}^{3}\text{div}\Pi_{\tilde{r},T}^{2,-}U = 
\Pi_{\tilde{r},T}^{3}\text{div}U.
\end{equation*}
Here $\Pi_{\tilde{r},T}^{3}$ is the orthogonal projection operator from 
$L^{2}(T;\mathbb{V})$ onto $\mathcal{P}_{\tilde{r}}\Lambda^{3}(T;\mathbb{V})$.
\end{lemma}

\begin{proof}
According to the definition of $\Pi_{\tilde{r},T}^{2,-}$, we have 
$(I-\Pi_{\tilde{r},T}^{2,-})\Pi_{\tilde{r},T}^{2,-}U=0$ for 
any $U\in H^{1}(T;\mathbb{M})$. So it is sufficient to show that 
$\Pi_{\tilde{r},T}^{3}\text{div}U=0$ for any $U\in H^{1}(T;\mathbb{M})$ 
with $\Pi_{\tilde{r},T}^{2,-}U=0$.

Now, we choose $U\in H^{1}(T;\mathbb{M})$ with $\Pi_{\tilde{r},T}^{2,-}U=0$. 
We only need to show that $\int_{T}\text{div}U\cdot\overline{\eta}=0$ for 
any $\overline{\eta}\in\mathcal{P}_{\tilde{r}(T)}(T;\mathbb{V})$. Obviously, 
we can choose $\boldsymbol{c}\in\mathbb{R}^{3}$ such that $\overline{\eta}=
\eta+\boldsymbol{c}$, where $\eta\in\mathcal{P}_{\tilde{r}(T)}(T;\mathbb{V})/\mathbb{R}$. 
Then we have 
\begin{equation*}
\int_{T}\text{div}U\cdot\overline{\eta}d\boldsymbol{x} = \int_{T}\text{div}U\cdot\eta d\boldsymbol{x}
+\int_{T}\text{div}U\cdot\boldsymbol{c}d\boldsymbol{x} = \int_{T}\text{div}U\cdot\eta d\boldsymbol{x}
+\int_{\partial T}(U\cdot\boldsymbol{n})\cdot\boldsymbol{c}ds.
\end{equation*}
By (\ref{op2_div_phy}),(\ref{op2_face_phy}) and the face that $\Pi_{\tilde{r},T}^{2,-}U=0$, 
we have $\int_{T}\text{div}U\cdot\overline{\eta}d\boldsymbol{x}=0$.
This implies that $\Pi_{\tilde{r},T}^{3}\text{div}\Pi_{\tilde{r},T}^{2,-}U = 
\Pi_{\tilde{r},T}^{3}\text{div}U$ for any $U\in H^{1}(T;\mathbb{M})$.
\end{proof}

\bigskip
\begin{lemma}
\label{cd_S1_phy}
For any $W\in H^{1}(T;\mathbb{M})$, we have 
\begin{equation*}
\Pi_{\tilde{r},T}^{2,-}S_{1}\Pi_{\tilde{r},T}^{1,-}W = 
\Pi_{\tilde{r},T}^{2,-}S_{1}W.
\end{equation*}
\end{lemma}

\begin{proof}
According to the definition of $\Pi_{\tilde{r},T}^{1,-}$, we have 
\begin{equation*}
(I-\Pi_{\tilde{r},T}^{1,-})\Pi_{\tilde{r},T}^{1,-}W=0,\quad W\in H^{1}(T;\mathbb{M}). 
\end{equation*}
So it is sufficient to show that 
$\Pi_{\tilde{r},T}^{2,-}S_{1}W=0$ for any $W\in H^{1}(T;\mathbb{M})$ 
with $\Pi_{\tilde{r},T}^{1,-}W=0$.

Now, we choose $W\in H^{1}(T;\mathbb{M})$ with $\Pi_{\tilde{r},T}^{1,-}W=0$.
By (\ref{op1_div_phy}) and (\ref{op1_aux_phy}), we have, 
\begin{equation*}
\int_{T}\text{div}S_{1}W\cdot\eta d\boldsymbol{x}=0,
\quad\eta\in\mathcal{P}_{\tilde{r}(T)}(T;\mathbb{V})/\mathbb{R};
\end{equation*}
\begin{equation*}
\int_{T}S_{1}W:[A\hat{\boldsymbol{h}}_{\tilde{r},i}(\hat{\boldsymbol{x}},t_{r})A^{-1}]
d\boldsymbol{x}=0,\quad 1\leq i\leq k.
\end{equation*} 

In order to demonstrate that $\Pi_{\tilde{r},T}^{2,-}S_{1}W=0$, we only need to show that
\begin{equation*}
\int_{F}[S_{1}W\cdot\boldsymbol{n}]\cdot
\mu ds=0, \quad F\in\triangle_{2}(T),\quad
\mu\in\mathcal{P}_{\tilde{r}(F)}(F;\mathbb{V}). 
\end{equation*} 

According to the definition of $S_{1}$, we have,
\begin{align*}
S_{1}W\cdot\boldsymbol{n}= &
\left[\begin{array}{ccc}
-w_{22}-w_{33} & w_{21} & w_{31} \\
w_{12} & -w_{11}-w_{33} & w_{32} \\
w_{13} & w_{23} & -w_{11}-w_{22}
\end{array}\right]
\cdot\boldsymbol{n} \\
= & \left[\begin{array}{c}
(n_{2}w_{21}-n_{1}w_{22})+(n_{3}w_{31}-n_{1}w_{33})\\
-(n_{2}w_{11}-n_{1}w_{12})+(n_{3}w_{32}-n_{2}w_{33})\\
-(n_{3}w_{11}-n_{1}w_{13})-(n_{3}w_{22}-n_{2}w_{23})
\end{array}\right].
\end{align*}

Consequently, for any $\mu\in\mathcal{P}_{\tilde{r}(F)}(F;\mathbb{V})$, we have, 
\begin{align*}
[S_{1}W\cdot\boldsymbol{n}]\cdot\mu = & 
\left(W\cdot\left[\begin{array}{c}
n_{2}\\-n_{1}\\0
\end{array}\right]\right)\cdot
\left[\begin{array}{c}
-\mu_{2}\\ \mu_{1}\\ 0
\end{array}\right]\\
& \quad + \left(W\cdot\left[\begin{array}{c}
n_{3}\\0\\-n_{1}
\end{array}\right]\right)\cdot
\left[\begin{array}{c}
-\mu_{3}\\ 0\\ \mu_{1}
\end{array}\right]
+\left(W\cdot\left[\begin{array}{c}
0 \\ n_{3}\\ -n_{2}
\end{array}\right]\right)\cdot
\left[\begin{array}{c}
0 \\ -\mu_{3}\\ \mu_{2}
\end{array}\right].
\end{align*}

By (\ref{op1_face_phy}) and the fact that $\Pi_{\tilde{r},T}^{1,-}W=0$, 
we conclude that 
\begin{equation*}
\int_{F}[S_{1}W\cdot\boldsymbol{n}]\cdot
\mu ds=0, \quad F\in\triangle_{2}(T),\quad 
\mu\in\mathcal{P}_{\tilde{r}(F)}(F;\mathbb{V}). 
\end{equation*} 
Consequently, $\Pi_{\tilde{r},T}^{2,-}S_{1}W=0$.
\end{proof}

\bigskip
\begin{lemma}
\label{op2_op1_bound_phy}
There exists $c>0$ such that, for any $U,W\in H^{1}(T;\mathbb{M})$, 
\begin{equation}
\Vert\Pi_{\tilde{r},T}^{2,-}U\Vert_{L^{2}(T;\mathbb{M})}\leq c\Vert U\Vert_{H^{1}(T;\mathbb{M})};
\label{op2_bound_phy}
\end{equation}
\begin{equation}
\Vert\text{curl}\Pi_{\tilde{r},T}^{1,-}W\Vert_{L^{2}(T;\mathbb{V})}\leq c
(h_{T}^{-1}\Vert W\Vert_{L^{2}(T;\mathbb{M})}+\Vert W\Vert_{H^{1}(T;\mathbb{M})}).
\label{op1_bound_phy}
\end{equation}
Here $h_{T}$ is the outer diameter of $T$, and $c$ is independent of $T$. The constant $c$ 
may depend upon the ratio of outer and inner diameters of $T$.
\end{lemma}

\begin{proof}
(\ref{op2_bound_phy},\ref{op1_bound_phy}) are obtained by standard scaling techniques.
The proof for (\ref{op2_bound_phy}) is the same as that 
for Lemma $20$ in \cite{QD:2009:MMEW}. 
The proof for (\ref{op1_bound_phy}) is the same as that 
for Lemma $29$ in \cite{QD:2009:MMEW}.
\end{proof}

\subsection{Projection operators on tetrahedral meshes}
As we stated at the beginning of this section, we use standard assumptions for 
regular meshes.
This means that the ratio between outer diameter and inner diameter of any tetrahedron 
in any 
mesh has a uniform upper bound.
We are going to extend operators $\Pi_{\tilde{r},T}^{2,-}$ 
and $\Pi_{\tilde{r},T}^{1,-}$ now to the 
whole mesh $\mathcal{T}_{h}$ in such a way that 
they make (\ref{commuting_diagram2},\ref{commuting_diagram3}) 
commute.

\begin{definition}
\label{op2_op1_global}
We define mappings $\Pi_{\tilde{r},h}^{2,-}:H^{1}(\Omega;\mathbb{M})\rightarrow
\mathcal{P}_{\tilde{r}+1}^{-}\Lambda^{2}(\mathcal{T}_{h};\mathbb{V})$ and 
$\Pi_{\tilde{r},h}^{1,-}:H^{1}(\Omega;\mathbb{M})\rightarrow\mathcal{P}_{\tilde{r}+2}^{-}
\Lambda^{1}(\mathcal{T}_{h};\mathbb{V})$ by
\begin{equation*}
(\Pi_{\tilde{r},h}^{2,-}U)|_{T}=\Pi_{\tilde{r},T}^{2,-}(U|_{T});\quad
(\Pi_{\tilde{r},h}^{1,-}W)|_{T}=\Pi_{\tilde{r},T}^{1,-}(W|_{T}).
\end{equation*}
Here $T\in\triangle_{3}(\mathcal{T}_{h})$, and $U,W\in H^{1}(\Omega;\mathbb{M})$.
\end{definition}

\bigskip
\begin{lemma}
For any $U,W\in H^{1}(\Omega;\mathbb{M})$, $\Pi_{\tilde{r},h}^{2,-}U\in
\mathcal{P}_{\tilde{r}+1}^{-}\Lambda^{2}(\mathcal{T}_{h};\mathbb{V})$ 
and $\Pi_{\tilde{r},h}^{1,-}W\in\mathcal{P}_{\tilde{r}+2}^{-}\Lambda^{1}
(\mathcal{T}_{h};\mathbb{V})$. And we have 
\begin{equation}
\Pi_{\tilde{r},h}^{3}\text{div}\Pi_{\tilde{r},h}^{2,-}U = \Pi_{\tilde{r},h}^{3}\text{div}U;\quad
\Pi_{\tilde{r},h}^{2,-}S_{1}\Pi_{\tilde{r},h}^{1,-}W = \Pi_{\tilde{r},h}^{2,-}S_{1}W.
\label{cd_div_S1_global}
\end{equation}

And there exists a constant $c>0$, which is independent of $\mathcal{T}_{h},U,W$ , so that 
\begin{equation}
\Vert\Pi_{\tilde{r},h}^{2,-}U\Vert_{L^{2}(\Omega;\mathbb{M})}\leq c\Vert U\Vert_{H^{1}(\Omega;\mathbb{M})};
\label{op2_bound_global}
\end{equation}
\begin{equation}
\Vert\text{curl}\Pi_{\tilde{r},h}^{1,-}W|_{T}\Vert_{L^{2}(T;\mathbb{V})}\leq c
(h_{T}^{-1}\Vert W|_{T}\Vert_{L^{2}(T;\mathbb{M})}+\Vert W|_{T}\Vert_{H^{1}(T;\mathbb{M})}).
\label{op1_bound_global}
\end{equation}
Here $T\in\triangle_{3}(\mathcal{T}_{h})$, and $h_{T}$ is the outer diameter of $T$.
\end{lemma}

\begin{proof}
This is by definitions of $\Pi_{\tilde{r},h}^{2,-}$ and $\Pi_{\tilde{r},h}^{1,-}$, 
Lemma \ref{cd_div_phy}, 
Lemma \ref{cd_S1_phy}, and Lemma \ref{op2_op1_bound_phy}.
\end{proof}

\bigskip
To remove the $h_{T}^{-1}$ factor in (\ref{op1_bound_global}), we introduce a Clement-type 
interpolant $R_{h}$ mapping $H^{1}(\Omega;\mathbb{M})$ into continuous piece-wise linear 
$M$-valued function on $\mathcal{T}_{h}$ (The operator $\Pi_{h}^{0}$ in Theorem $5.1$ of 
\cite{CB:1989:OIC}, using example $1$ in \cite{CB:1989:OIC}). Then there exists 
a constant $c>0$ 
such that $\forall W\in H^{1}(\Omega;\mathbb{M}),T\in\mathcal{T}_{h}$, we have that
\[
\Vert W-R_{h}W\Vert_{L^{2}(T;\mathbb{M})}\leq ch_{T}\Vert W\Vert_{H^{1}(T;\mathbb{M})};
\Vert W-R_{h}W\Vert_{H^{1}(T;\mathbb{M})}\leq c\Vert W\Vert_{H^{1}(\Sigma_{T};\mathbb{M})}.
\]
Here $\Sigma_{T}:=\bigcup_{T^{'}\in\mathcal{T}_{h}:T^{'}\cap T\neq\emptyset}T^{'}$.
Then we follow \cite{AW:2002:MME} and define $\overline{\Pi}_{\tilde{r},h}^{1,-}:=
\Pi_{\tilde{r},h}^{1,-}(I-R_{h})+R_{h}$.

\begin{lemma}
\label{overline_op1_global}
$\overline{\Pi}_{\tilde{r},h}^{1,-}$ maps from $H^{1}(\Omega;\mathbb{M})$ into 
$\mathcal{P}_{\tilde{r}+2}^{-}\Lambda^{1}(\mathcal{T}_{h};\mathbb{V})$.
\begin{equation}
\Pi_{\tilde{r},h}^{2,-}S_{1}\overline{\Pi}_{\tilde{r},h}^{1,-}W = 
\Pi_{\tilde{r},h}^{2,-}S_{1}W,\quad W\in H^{1}(\Omega;\mathbb{M}).
\label{cd_S1_global_overline}
\end{equation} 

And there exists a constant $c>0$ such that for any $W\in H^{1}(\Omega;\mathbb{M})$,
\begin{equation}
\Vert\text{curl}\overline{\Pi}_{\tilde{r},h}^{1,-}W\Vert_{L^{2}(\Omega;\mathbb{V})}
\leq c\Vert W\Vert_{H^{1}(\Omega;\mathbb{M})}.
\label{op1_bound_global_overline}
\end{equation} 
\end{lemma}

\begin{proof}
Since $R_{h}$ maps $H^{1}(\Omega;\mathbb{M})$ into continuous piece-wise linear 
$M$-valued function on $\mathcal{T}_{h}$, we have $\overline{\Pi}_{\tilde{r},h}^{1,-}$ 
maps from $H^{1}(\Omega;\mathbb{M})$ into 
$\mathcal{P}_{\tilde{r}+2}^{-}\Lambda^{1}(\mathcal{T}_{h};\mathbb{V})$.
The proof for (\ref{cd_S1_global_overline},\ref{op1_bound_global_overline}) is straightforward.
\end{proof}

\section{Stability of the finite element discretization}

We will use the following well-known result from partial differential
equations, see \cite{GR:1986:FENS}.

\begin{lemma}
\label{PDE_lemma}Let $\Omega$ be a bounded domain in $\mathbb{R}^{3}$ with 
a Lipschitz boundary. Then, for all $\mu\in L^{2}(\Omega)$, there exists 
$\eta\in H^{1}(\Omega;\mathbb{V})$ satisfying $\text{div}\eta=\mu$. If, in 
addition, $\int_{\Omega}\mu d\boldsymbol{x}=0$, then we can choose 
$\eta\in\mathring{H}^{1}(\Omega;\mathbb{V})$.
\end{lemma}

\begin{remark}
The domain $\Omega$ need {\em not}  be contractible.
\end{remark}

\bigskip 
The main result of this paper is the following theorem. In the proof we follow
the lines of proof of 
Theorem 9.1 in \cite{Falk:2008:FME}, Theorem 7.1 in
\cite{AFW:2007:MMEW} and Theorem 11.4 in \cite{AFW:2006:ECH}. The main
difference is in the use of our operator $\overline{\Pi}_{\tilde{r},h}^{1,-}$ in place of the 
operator $\tilde{\Pi}_{h}^{n-2}$ from \cite{Falk:2008:FME}.

\begin{theorem}
\label{stable_condition_theorem}Let $\Omega$ be a bounded polyhedral domain in
$\mathbb{R}^{3}$ with a Lipschitz boundary. We assume that the meshes are regular.
Then for any $(\omega,\mu)\in 
\mathcal{P}_{\tilde{r}}\Lambda^{3}(\mathcal{T}_{h};\mathbb{V})\times 
\mathcal{P}_{\tilde{r}}\Lambda^{3}(\mathcal{T}_{h};\mathbb{V})$, there exists 
$\sigma\in\mathcal{P}_{\tilde{r}+1}\Lambda^{2}(\mathcal{T}_{h};\mathbb{V})$ 
such that $\text{div}\sigma =\mu$, $-\Pi_{\tilde{r},\mathcal{T}_{\tilde{r},h}}^{3}
S_{2}\sigma=\omega$. And we have
\begin{equation}
\Vert\sigma\Vert _{H(\text{div},\Omega;\mathbb{M})}\leq c(\Vert
\omega\Vert_{L^{2}(\Omega;\mathbb{V})}+\Vert\mu\Vert_{L^{2}(\Omega;\mathbb{V})}), 
\label{stable_result}
\end{equation}
where the constant $c$ is independent of $\omega,\mu$ and h, but it may depend upon
$\max_{T\in\triangle_{3}(\mathcal{T}_{h})}\tilde{r}(T)$.
\end{theorem}

\begin{proof}
We want to show that Brezzi stability conditions (\ref{S1_condition}
),(\ref{S2_condition}) are satisfied. The condition (\ref{S1_condition}) is
obviously satisfied since, by construction, $\text{div}\mathcal{P}_{\tilde{r}+1}\Lambda^{2}
(\Omega;\mathbb{V})\subset\mathcal{P}_{\tilde{r}}\Lambda^{3}(\Omega;\mathbb{V})$ and the
fact that $A$ is coercive.

Now we only need to prove that the condition (\ref{S2_condition}) is satisfied as well.

(1) By Lemma \ref{PDE_lemma}, we can find $\eta\in H^{1}(
\Omega;\mathbb{M})$ with $\text{div}\eta=\mu$ and $\Vert\eta\Vert
_{H^{1}(\Omega;\mathbb{M})}\leq c\Vert\mu\Vert_{L^{2}(\Omega;\mathbb{V})}$.

(2) Since $\omega+\Pi_{\tilde{r},h}^{3}S_{2}\Pi_{\tilde{r},h}^{2}\eta\in L^{2}
(\Omega;\mathbb{V})$, we can apply Lemma \ref{PDE_lemma}
again to find $\tau\in H^{1}(\Omega;\mathbb{M})$
with $\text{div}\tau=\omega+\Pi_{\tilde{r},h}^{3}S_{2}\Pi_{\tilde{r},h}^{2}\eta$ and
\[
\Vert\tau\Vert _{H^{1}(\Omega;\mathbb{M})}\leq c(\Vert\omega\Vert_{L^{2}(\Omega;\mathbb{V})}
+\Vert\Pi_{\tilde{r},h}^{3}S_{2}\Pi_{\tilde{r},h}^{2}\eta\Vert_{L^{2}(\Omega;\mathbb{V})})  .
\]

(3) Since $S_{1}$ is an isomorphism from $H^{1}(\Omega;\mathbb{M})$ to 
$H^{1}(\Omega;\mathbb{M})$, we have $\varrho\in H^{1}(\Omega;\mathbb{M})$ 
with $S_{1}\varrho=\tau$, and $\Vert\varrho\Vert _{H^{1}(\Omega;\mathbb{M})}\leq 
c\Vert\tau\Vert_{H^{1}(\Omega;\mathbb{M})}$.

(4) Define $\sigma=\text{curl}\overline{\Pi}_{\tilde{r},h}^{1,-}\varrho+\Pi_{\tilde{r},h}^{2}
\eta\in\mathcal{P}_{\tilde{r}+1}\Lambda^{2}(\mathcal{T}_{h};\mathbb{V})$. According to 
Lemma $14$ in \cite{QD:2009:MMEW}, $\text{curl}\overline{\Pi}_{\tilde{r},h}^{1,-}\varrho
\in\mathcal{P}_{\tilde{r}+1}\Lambda^{2}(\mathcal{T}_{h};\mathbb{V})$. So we have 
$\sigma\in\mathcal{P}_{\tilde{r}+1}\Lambda^{2}(\mathcal{T}_{h};\mathbb{V})$.

(5) From step (4), (\ref{cd1}), step (1), and the fact that
$\Pi_{\tilde{r},h}^{3}$ is a projection, we have
\[
\text{div}\sigma=\text{div}\Pi_{\tilde{r},h}^{2}\eta=\Pi_{\tilde{r},h}^{3}\text{div}\eta=
\Pi_{\tilde{r},h}^{3}\mu=\mu.
\]

(6) Also from step (4),
\[
-\Pi_{\tilde{r},h}^{3}S_{2}\sigma=-\Pi_{\tilde{r},h}^{3}S_{2}\text{curl}\overline{\Pi}_{\tilde{r},h}^{1,-}
\varrho-\Pi_{\tilde{r},h}^{3}S_{2}\Pi_{\tilde{r},h}^{2}\eta.
\]
Applying, in order, Lemma \ref{S1_S2_commuting_diagram}, (\ref{cd_div_S1_global}), 
(\ref{cd_S1_global_overline}), step (3), (\ref{cd_div_S1_global}), step (2), and 
the fact that $\Pi_{\tilde{r},h}^{3}$ is a projection, we obtain
\begin{align*}
-\Pi_{\tilde{r},h}^{3}S_{2}\text{curl}\overline{\Pi}_{\tilde{r},h}^{1,-}\varrho &  
= \Pi_{\tilde{r},h}^{3}\text{div}S_{1}\overline{\Pi}_{\tilde{r},h}^{1,-}\varrho
=\Pi_{\tilde{r},h}^{3}\text{div}\Pi_{\tilde{r},h}^{2,-}
S_{1}\overline{\Pi}_{\tilde{r},h}^{1,-}\varrho\\
&  =\Pi_{\tilde{r},h}^{3}\text{div}\Pi_{\tilde{r},h}^{2,-}S_{1}\varrho=\Pi_{\tilde{r},h}^{3}
\text{div}\Pi_{\tilde{r},h}^{2,-}\tau=\Pi_{\tilde{r},h}^{3}\text{div}\tau\\
&  =\Pi_{\tilde{r},h}^{3}(\omega+\Pi_{\tilde{r},h}^{3}S_{2}\Pi_{\tilde{r},h}^{2}
\eta)  =\omega+\Pi_{\tilde{r},h}^{3}S_{2}\Pi_{\tilde{r},h}^{2}\eta.
\end{align*}
Combining, we have $-\Pi_{\tilde{r},h}^{3}S_{2}\sigma=\omega$.

(7) Finally, we prove the norm bound. From the boundedness of $S_{2}$ in
$L^{2}$, (\ref{cd1bound}), and step (1),
\[
\Vert\Pi_{\tilde{r},h}^{3}S_{2}\Pi_{\tilde{r},h}^{2}\eta\Vert_{L^{2}(\Omega;\mathbb{V})}\leq
c\Vert S_{2}\Pi_{\tilde{r},h}^{2}\eta\Vert_{L^{2}(\Omega;\mathbb{V})}\leq c\Vert
\Pi_{\tilde{r},h}^{2}\eta\Vert_{L^{2}(\Omega;\mathbb{M})}\leq c\Vert\eta\Vert _{H^{1}
(\Omega;\mathbb{M})}\leq c\Vert\mu\Vert_{L^{2}(\Omega;\mathbb{V})} .
\]
Combining with the bounds in step (3) and (2), this gives $\Vert\varrho\Vert_{H^{1}
(\Omega;\mathbb{M})}\leq c(\Vert\omega\Vert_{L^{2}(\Omega;\mathbb{V})}+\Vert
\mu\Vert_{L^{2}(\Omega;\mathbb{V})})$. From (\ref{cd_S1_global_overline}), we then have
$\Vert\text{curl}\overline{\Pi}_{\tilde{r},h}^{1,-}\varrho\Vert_{L^{2}(\Omega;\mathbb{V})}
\leq c\Vert\varrho
\Vert_{H^{1}(\Omega;\mathbb{M})}\leq c(\Vert\omega\Vert_{L^{2}(\Omega;\mathbb{V})}
+\Vert\mu\Vert_{L^{2}(\Omega;\mathbb{V})})$. From (\ref{cd1bound}) and the bound in Step
(1), $\Vert\Pi_{\tilde{r},h}^{2}\eta\Vert_{L^{2}(\Omega;\mathbb{M})}\leq c\Vert
\eta\Vert_{H^{1}(\Omega;\mathbb{M})}\leq c\Vert\mu\Vert_{L^{2}(\Omega;\mathbb{V})}$. In 
view of the definition of $\sigma$, these two last bounds imply that $\Vert
\sigma\Vert_{L^{2}(\Omega;\mathbb{M})}\leq c(\Vert\omega\Vert_{L^{2}(\Omega;\mathbb{V})}
+\Vert\mu\Vert_{L^{2}(\Omega;\mathbb{V})})$, while $\Vert\text{div}\sigma\Vert_{L^{2}
(\Omega;\mathbb{V})} =\Vert\mu\Vert_{L^{2}(\Omega;\mathbb{V})}$ by Step (5), and thus we 
have the desired bound (\ref{stable_result}).
\end{proof}

\bigskip
We have thus verified the stability conditions (\ref{S1_condition}) and
(\ref{S2_condition}), and so obtain the following quasi-optimal error estimate.

\begin{theorem}
Suppose $(\sigma,u,p)$ is the solution of the elasticity system
(\ref{weak_symmetry_S2}) and $(\sigma_{h},u_{h}
,p_{h})$ is the solution of discrete system
(\ref{weak_symmetry_discrete}), where the finite element
spaces satisfy the hypotheses of Theorem \ref{stable_condition_theorem}. 
We also assume that there is $r_{\max}\in\mathbb{N}$ such that for 
any $h>0$ and $f\in\Delta (\mathcal{T}_{h})$,$\tilde{r}(f)\leq r_{\max}$.
Then there is a constant $C$, independent of $h$, such that
\begin{align*}
& \Vert\sigma-\sigma_{h}\Vert_{H(\text{div},\Omega;\mathbb{M})}+\Vert u-u_{h}
\Vert_{L^{2}(\Omega;\mathbb{V})}+\Vert p-p_{h}\Vert_{L^{2}(\Omega;\mathbb{V})}\\
\leq & C\inf(\Vert\sigma-\tau\Vert_{H(\text{div},\Omega;\mathbb{M})}+\Vert u-v\Vert
_{L^{2}(\Omega;\mathbb{V})}+\Vert p-q\Vert_{L^{2}(\Omega;\mathbb{V})}),
\end{align*}
where the infimum is taken over all $\tau\in\mathcal{P}_{\tilde{r}+1}^{2}
(\mathcal{T}_{h};\mathbb{V}),v\in\mathcal{P}_{\tilde{r}}^{3}(\mathcal{T}_{h};\mathbb{V})$, 
and $q\in\mathcal{P}_{\tilde{r}}^{3}(\mathcal{T}_{h};\mathbb{V})$.
\end{theorem}

\section{Conclusions and future work}

In the paper, we have presented a generalization of Arnold-Falk-Winther (AFW) elements
to the case of elements of variable order for a three dimensional domain. The proof of 
stability is based on the use of some variant of projection based interpolation operators, 
and a specially designed operator $\overline{\Pi}_{\tilde{r},h}^{1,-}$ discussed in the text. 
We have proved the $h$-stability for meshes with variable order under the assumption 
that there is an uniform upper bound on the highest polynomial order used. 

We plan to continue the research on several fronts. On the numerical side,
we intend to implement and test the $hp$-adaptive algorithm based on the
coarse/fine grid paradigm. The code will be applied to a detailed study
of problems with large material contrast including the streamer problem, 
discussed in \cite{QD:2009:MMEW}. The results obtained using the 
AFW elements will be compared with results obtained using the classical 
$H^1$-conforming elements, in terms of memory use and CPU time. 

On the theoretical side, we will attempt to prove $p$-stability and, ultimately, 
the $hp$-stability of the method.


\begin{thebibliography}{10}

\bibitem{AT:1979:EFEE}
{\sc M.~Amara and J.~M. Thomas}, {\em Equilibrium finite elements for the
  linear elastic problem}, Numer. Math., 33 (1979), pp.~367--383.

\bibitem{ABD:1984:MFEPE}
{\sc D.~N. Arnold, F.~Brezzi, and J.~Douglas}, {\em Peers: a new mixed finite
  element for plane elasticity}, Japan J. Appl. Math., 1 (1984), pp.~347--367.

\bibitem{AF:1988:NMFEE}
{\sc D.~N. Arnold and R.~S. Falk}, {\em A new mixed formulation for
  elasticity}, Numer. Math., 53 (1988), pp.~13--30.

\bibitem{AFW:2006:ECH}
{\sc D.~N. Arnold, R.~S. Falk, and R.~Winther}, {\em Finite element exterior
  calculus, homological techniques, and applications}, Acta Numer.,  (2006),
  pp.~1--155.

\bibitem{AFW:2007:MMEW}
\leavevmode\vrule height 2pt depth -1.6pt width 23pt, {\em Mixed finite element
  methods for linear elasticity with weakly imposed symmetry}, Mathematics of
  Computation, 76 (2007), pp.~1699--1723.

\bibitem{AW:2002:MME}
{\sc D.~N. Arnold and R.~Winther}, {\em Mixed finite elements elasticity},
  Numerische Mathematik, 92 (2002), pp.~401--419.

\bibitem{CB:1989:OIC}
{\sc C.~Bernardi}, {\em Optimal finite-element interpolation on curved
  domains}, SIAM Journal on Numerical Analysis, 26 (1989), pp.~1212--1240.

\bibitem{CGG:2009:NEEMW}
{\sc B.~Cockburn, J.~Gopalakrishnan, and J.~Guzman}, {\em A new elasticity
  element made for enforcing weak stress symmetry}, Mathematics of Computation,
   (2009).

\bibitem{FdV:1975:SFA}
{\sc B.~M.~F. de~Veubeke}, {\em Stress function approach}, Proc. of the World
  Congress on Finite Element Methods in Structural Mechanics, 5 (1975),
  pp.~J.1--J.51.

\bibitem{Demkowicz:2006:HPAFE}
{\sc L.~Demkowicz}, {\em Computing with hp-adaptive finite elements Volume I
  One and Two Dimensional Elliptic and Maxwell Problems}, Chapman Hall CRC,
  2006.

\bibitem{Demkowicz:2007:HPAFE2}
{\sc L.~Demkowicz, J.~Kurtz, D.~Pardo, M.~Paszynski, W.~Rachowicz, and
  A.~Zdunek}, {\em Computing with hp-adaptive finite elements Volume II
  Three-Dimensional Elliptic and Maxwell Problems with Applications}, Chapman
  Hall CRC, 2007.

\bibitem{Falk:2008:FME}
{\sc R.~S. Falk}, {\em Finite element methods for linear elasticity}, in
  Lecture Notes in Mathematics, Springer-Verlag, 2008, pp.~160--194.

\bibitem{FarhloulFortin:1997:DHES}
{\sc M.~Farhloul and M.~Fortin}, {\em Dual hybrid methods for the elasticity
  and the {S}tokes problem: a unified approach}, Numer. Math., 76 (1997),
  pp.~419--440.

\bibitem{GR:1986:FENS}
{\sc V.~Girault and P.~A. Raviart}, {\em Finite element methods for
  Navier-Stokes equations: theory and algorithms}, Springer-Verlag, Berlin; New
  york, 1986.

\bibitem{Morley:1989:MFEE}
{\sc M.~E. Morley}, {\em A family of mixed finite elements for linear
  elasticity}, Numer. Math., 55 (1989), pp.~633--666.

\bibitem{ODENREDDY:1976:VMTM}
{\sc J.~T. Oden and J.~N. Reddy}, {\em Variational methods in theoretical
  mechanics}, Springer-Verlag, 1976.

\bibitem{QD:2009:MMEW}
{\sc W.~Qiu and L.~Demkowicz}, {\em Mixed $\boldsymbol{hp}$-finite element
  method for linear elasticity with weakly imposed symmetry}, Comput. Methods
  Appl. Mech. Engrg., 198 (2009), pp.~3682--3701.

\bibitem{SteinRolfes:1990:SOMFEPE}
{\sc E.~Stein and R.~Rolfes}, {\em Mechanical conditions for stability and
  optimal convergence of mixed finite elements for linear plane elasticity},
  Comput. Methods Appl. Mech. Engrg., 84 (1990), pp.~77--95.

\bibitem{Stenberg:1986:COMME}
{\sc R.~Stenberg}, {\em On the construction of optimal mixed finite element
  methods for the linear elasticity problem}, Numer. Math., 48 (1986),
  pp.~447--462.

\bibitem{Stenberg:1988:FMME}
\leavevmode\vrule height 2pt depth -1.6pt width 23pt, {\em A family of mixed
  finite elements for the elasticity problem}, Numer. Math., 53 (1988),
  pp.~513--538.

\bibitem{Stenberg:1988:TLOMME}
\leavevmode\vrule height 2pt depth -1.6pt width 23pt, {\em Two low-order mixed
  methods for the elasticity problem}, The Mathematics of Finite Elements and
  Applications, VI,  (1988), pp.~271--280.

\end{thebibliography}
\end{document}